\documentclass[a4paper,leqno,10t]{article}
  
\usepackage[off]{auto-pst-pdf}
\usepackage{amsmath,amsfonts,amssymb}
\usepackage{fancyhdr} 
\usepackage{epstopdf}
\usepackage{latexsym}
\usepackage{mathrsfs,amssymb} 
\usepackage[intlimits]{empheq} 
\usepackage[arrow, matrix, curve]{xy}
\usepackage{amsmath}
\usepackage{amssymb}
\usepackage{empheq}
\usepackage{mathrsfs}
\usepackage[noindentafter]{titlesec}
\usepackage{blindtext}
\usepackage[a4paper]{geometry}
\usepackage[latin1]{inputenc}
\usepackage[T1]{fontenc}
\usepackage{amsthm}
\usepackage{dsfont}
\usepackage{pgf}
\usepackage{graphicx}
\usepackage[nottoc]{tocbibind}
\usepackage[linktocpage=true,colorlinks=false]{hyperref}
\usepackage{xcolor}
\usepackage{pfnote}
\usepackage{fnpos} 
\usepackage{color}
 \usepackage{fancyhdr,blindtext}
\fancypagestyle{plain}{%
\fancyhf{}}

\renewcommand{\sectionmark}[1]%
{\markboth{#1}{}}

\makeindex

\theoremstyle{definition} 

\newtheorem{ex}{\bfseries \upshape Example}[section]
\newtheorem{dfn}[ex]{\bfseries \upshape Definition}
\newtheorem{rem}[ex]{\bfseries \upshape Remark}

\newtheorem{qu}[ex]{\bfseries \upshape Question}

\theoremstyle{plain}
\newtheorem{prop}[ex]{\bfseries \upshape Proposition}
\newtheorem{lem}[ex]{\bfseries \upshape Lemma}
\newtheorem{theo}{\bfseries \upshape Theorem}
\newtheorem{cor}{\bfseries \upshape Corollary}
\newenvironment{pr}{\begin{proof}[{\bf Proof}]}{\end{proof}}

\begin{document}
\title{{ \bf Uniform rigidity sequences for weak mixing diffeomorphisms on $\mathbb{T}^2$}}
\author{{\sc Philipp Kunde}}

\maketitle
 
\begin{abstract} 
In this paper we will show that if a sequence of natural numbers satisfies a certain growth rate, then there is a weak mixing diffeomorphism on $\mathbb{T}^2$ that is uniformly rigid with respect to that sequence. The proof is based on a quantitative version of the Anosov-Katok-method with explicitly defined conjugation maps and the constructions are done in the $C^{\infty}$-topology as well as in the real-analytic topology.
\end{abstract}


\section{Introduction}

In \cite{GM} the notion of uniform rigidity was introduced as the topological analogue of rigidity in ergodic theory:
\begin{dfn}
\begin{enumerate}
	\item Let $T$ be an invertible measure-preserving transformation of a non-atomic probability space $\left(X, \mathcal{B}, \mu\right)$. $T$ is called rigid if there exists an increasing sequence $\left(n_m\right)_{m \in \mathbb{N}}$ of natural numbers such that the powers $T^{n_m}$ converge to the identity in the strong operator topology as $m \rightarrow \infty$, i.e. $\left\|f \circ T^{n_m} - f \right\|_2 \rightarrow 0$ as $m\rightarrow \infty$ for all $f \in L^2\left(X,\mu\right)$. So rigidity along a sequence $\left(n_m\right)_{m \in \mathbb{N}}$ implies $\mu\left(T^{n_m}A \cap A\right) \rightarrow \mu\left(A\right)$ as $m\rightarrow \infty$ for all $A \in \mathcal{B}$.
	\item Let $\left(X, \mathcal{B},\mu\right)$ be a Lebesgue probability space, where $X$ is a compact metric space with metric $d$. A measure-preserving homeomorphism $T: X\rightarrow X$ is called uniformly rigid if there exists an increasing sequence $\left(n_m\right)_{m \in \mathbb{N}}$ of natural numbers such that $d_u\left(T^{n_m}, id\right) \rightarrow 0$ as $m \rightarrow \infty$, where $d_u\left(S,T\right) = d_{0}\left(S,T\right)+d_{0}\left(S^{-1},T^{-1}\right)$ with $d_{0}\left(S,T\right) \coloneqq \sup_{x \in X}d\left(S\left(x\right), T\left(x\right)\right)$ is the uniform metric on the group of measure-preserving homeomorphisms on $X$.
\end{enumerate}
\end{dfn}
\begin{rem}
Uniform rigidity implies rigidity. In \cite{Ya}, example 3.1, an example of a rigid, but not uniformly rigid homeomorphism of $\mathbb{T}^2$ is presented. Thus, rigidity and uniform rigidity do not coincide on $\mathbb{T}^2$. 
\end{rem}
In \cite{JKLSS}, Proposition 4.1., it is shown that if an ergodic map is uniformly rigid, then any uniform rigidity sequence has zero density. Afterwards, the following question is posed:
\begin{qu} \label{qu:seq}
Which zero density sequences occur as uniform rigidity sequences for an ergodic transformation?
\end{qu}
Under some assumptions on the sequence $\left(n_m\right)_{m \in \mathbb{N}}$ measure-preserving transformations that are weak mixing and rigid along this sequence are constructed by a cutting and stacking method in \cite{BJLR}. Recall that a measure-preserving transformation $T:\left(X, \mathcal{B}, \mu\right)\rightarrow \left(X, \mathcal{B}, \mu\right)$ is called weak mixing if for all $A,B \in \mathcal{B}$: $\frac{1}{N}\sum^{N}_{n=1}\left| \mu\left(T^nA\cap B\right)-\mu\left(A\right)\cdot \mu\left(B\right)\right|\rightarrow 0$ as $N\rightarrow \infty$. \\
K. Yancey considered Question \ref{qu:seq} in the setting of homeomorphisms on $\mathbb{T}^2$ (see \cite{Ya}). Given a sufficient growth rate of the sequence she proved the existence of a weak mixing homeomorphism of $\mathbb{T}^2$ that is uniformly rigid with respect to this sequence: Let $\psi\left(x\right)=x^{x^3}$. If $\left(n_m\right)_{m \in \mathbb{N}}$ is an increasing sequence of natural numbers satisfying $\frac{n_{m+1}}{n_m}\geq \psi\left(n_m\right)$, there exists a weak mixing homeomorphism of $\mathbb{T}^2$ that is uniformly rigid with respect to $\left(n_m\right)_{m \in \mathbb{N}}$. \\
In this paper we start to examine this problem in the smooth category. The aimed diffeomorphisms are constructed with the aid of the so-called ``conjugation by approximation-method'' introduced in \cite{AK}. On every smooth compact connected manifold of dimension $m\geq 2$ admitting a non-trivial circle action $\mathcal{S} = \left\{S_t\right\}_{t \in \mathbb{S}^1}$ this method enables the construction of smooth diffeomorphisms with specific ergodic properties (e.g. weak mixing ones in \cite{AK}, section 5, or \cite{GK}) or non-standard smooth realizations of measure preserving systems (e.g. \cite{AK}, section 6, and \cite{FSW}). These diffeomorphisms are constructed as limits of conjugates $f_n = H_n \circ S_{\alpha_{n+1}} \circ H^{-1}_n$, where $\alpha_{n+1} = \alpha_n + \frac{1}{k_n \cdot l_n \cdot q^2_n} \in \mathbb{Q}$, $H_n = H_{n-1} \circ h_n$ and $h_n$ is a measure-preserving diffeomorphism satisfying $R_{\frac{1}{q_n}} \circ h_n = h_n \circ R_{\frac{1}{q_n}}$. In each step the conjugation map $h_n$ and the parameter $k_n$ are chosen such that the diffeomorphism $f_n$ imitates the desired property with a certain precision. Then the parameter $l_n$ is chosen large enough to guarantee closeness of $f_{n}$ to $f_{n-1}$ in the $C^{\infty}$-topology and so the convergence of the sequence $\left(f_n\right)_{n \in \mathbb{N}}$ to a limit diffeomorphism is provided. See \cite{FK} for more details and other results of this method. \\
As a starting point we use the construction of weak mixing diffeomorphisms on $\mathbb{T}^2$ undertaken in the real-analytic topology in \cite{FS} with the explicit conjugation maps
\begin{align*}
& \phi_n\left(\theta,r\right) = \left( \theta, r + q^2_n \cdot \cos\left(2\pi q_n \theta\right)\right), \\
& g_n\left(\theta,r\right) = \left( \theta + \left[nq^{\sigma}_{n}\right] \cdot r,r\right) \text{ with some $0<\sigma<\frac{1}{2}$}, \\
& h_n = g_n \circ \phi_n.
\end{align*} 
Here $\left[ \cdot \right]$ denotes the integer part of the number. Furthermore, let $\mathcal{R} = \left\{R_t\right\}_{t \in \mathbb{S}^1}$ denote the standard circle action on $\mathbb{T}^2$ comprising of the diffeomorphisms $R_t\left(\theta,r\right)= \left(\theta + t, r\right)$. Note that $h_n \circ R_{\frac{p_n}{q_n}} = R_{\frac{p_n}{q_n}} \circ h_n$. With the conjugation maps $H_n \coloneqq h_1 \circ ... \circ h_n$  we will define the diffeomorphism $f_n = H_n \circ R_{\alpha_{n+1}} \circ H^{-1}_{n}$. The sequence of rational numbers will be
\begin{equation*}
\alpha_{n+1}= \frac{p_{n+1}}{q_{n+1}}= \alpha_n - \frac{a_{n}}{q_n \cdot \tilde{q}_{n+1}},
\end{equation*}
where $a_n \in \mathbb{Z}$, $1\leq a_n \leq q_n$ is chosen in such a way that $\tilde{q}_{n+1} \cdot p_n \equiv a_n \mod q_n$. Therewith, we have $\left| \alpha_{n+1} - \alpha_n \right| \leq \frac{1}{\tilde{q}_{n+1}}$ and $\tilde{q}_{n+1} \cdot \alpha_{n+1} = \frac{\tilde{q}_{n+1} \cdot p_n}{q_n} - \frac{a_n}{q_n} \equiv 0 \mod 1$, which implies $f^{\tilde{q}_{n+1}}_{n} = \text{id}$. Hence, $\left(\tilde{q}_n\right)_{n \in \mathbb{N}}$ will be a rigidity sequence of $f= \lim_{n\rightarrow \infty} f_n$ under some restrictions on the closeness between $f_n$ and $f$ (see Remark \ref{rem:rigid}), which depend on the norms of the conjugation maps $H_i$ and the distances $\left|\alpha_{i+1} - \alpha_i \right| \leq \frac{1}{\tilde{q}_{i+1}}$ for every $i>n$. Thus, we have to estimate the norms $|||H_n|||_{n}$ carefully. This will yield the subsequent requirement on the number $\tilde{q}_{n+1}$ (see the end of section \ref{subsection:convA}): 
\begin{equation*}
\tilde{q}_{n+1} > \varphi_1\left(n\right) \cdot  \tilde{q}^{2 \cdot \left(\left(n+2\right) \cdot \left(n+1\right)^{n+1} +1\right)}_n,
\end{equation*}
where $\varphi_1\left(n\right) \coloneqq 2^{n} \cdot \left(n+1\right)! \cdot \left(\left(n + 2\right)!\right)^{\left(n+2\right)^{n-2} \cdot \left(n +1\right)} \cdot \left(2 \pi n\right)^{\left(n+2\right) \cdot \left(n+1\right)^{n+1}}$.
This is a sufficient condition on the growth rate of the rigidity sequence $\left(\tilde{q}_n\right)_{n \in \mathbb{N}}$ and we prove that $f$ is weak mixing using a criterion similar to that deduced in \cite{FS} (see section \ref{section:wm}). Consequently we obtain: 
\begin{theo} \label{theo:main}
Let $\varphi_1\left(n\right) \coloneqq 2^{n} \cdot \left(n+1\right)! \cdot \left(\left(n + 2\right)!\right)^{\left(n+2\right)^{n-2} \cdot \left(n +1\right)} \cdot \left(2 \pi n\right)^{\left(n+2\right) \cdot \left(n+1\right)^{n+1} }$. If $\left(\tilde{q}_n\right)_{n \in \mathbb{N}}$ is a sequence of natural numbers satisfying
\begin{equation*}
\tilde{q}_{n+1} \geq \varphi_1\left(n\right) \cdot \tilde{q}^{2 \cdot \left(\left(n+2\right) \cdot \left(n+1\right)^{n+1} +1\right)}_n,
\end{equation*}
there exists a weak mixing $C^{\infty}$-diffeomorphism of $\mathbb{T}^2$ that is uniformly rigid with respect to $\left(\tilde{q}_n\right)_{n \in \mathbb{N}}$.
\end{theo}
In section \ref{subsection:prKor} we conclude a rougher but more handsome statement: 
\begin{cor} \label{cor:Kor}
If $\left(\tilde{q}_n\right)_{n \in \mathbb{N}}$ is a sequence of natural numbers satisfying $\tilde{q}_1 \geq 108 \pi$ and $\tilde{q}_{n+1} \geq \tilde{q}^{\tilde{q}_n}_n$, then there exists a weak mixing $C^{\infty}$-diffeomorphism of $\mathbb{T}^2$ that is uniformly rigid with respect to $\left(\tilde{q}_n\right)_{n \in \mathbb{N}}$.
\end{cor}
We note that our requirement on the growth rate is less restrictive than the mentioned condition in \cite{Ya}, Theorem 1.5.. In fact, the proof in \cite{Ya} shows that a condition of the form $\frac{n_{m+1}}{n_m} \geq n^{4n^2_m+20}_m$ is sufficient for her construction of a weakly mixing homeomorphism, which is uniformly rigid along $\left(n_m\right)_{m \in \mathbb{N}}$. Our requirement on the growth rate is still weaker. \\ 
By the same approach we consider the problem in the real-analytic topology. In this setting we will deduce the following sufficient condition on the growth rate of the rigidity sequence $\left(\tilde{q}_n\right)_{n \in \mathbb{N}}$:
\begin{theo} \label{theo:main2}
Let $\rho>0$. If $\left(\tilde{q}_n\right)_{n \in \mathbb{N}}$ is a sequence of natural numbers satisfying $\tilde{q}_1 \geq \rho+1$ and
\begin{equation*}
\tilde{q}_{n+1} \geq 2^n \cdot 64\pi^2 \cdot n^2 \cdot \tilde{q}^{14}_n \cdot \exp\left(4\pi \cdot n \cdot \tilde{q}^6_n \cdot \exp\left(2 \pi \cdot \tilde{q}^4_n \cdot \left(1+n \cdot \tilde{q}_n\right)\right)\right),
\end{equation*}
there exists a weak mixing Diff$^{\omega}_{\rho}$-diffeomorphism of $\mathbb{T}^2$ that is uniformly rigid along the sequence $\left(\tilde{q}_n\right)_{n \in \mathbb{N}}$.
\end{theo}
Again, we derive from this a more convenient statement in section \ref{subsection:prKorA}:
\begin{cor} \label{cor:KorA}
If $\left(\tilde{q}_n\right)_{n \in \mathbb{N}}$ is a sequence of natural numbers satisfying $\tilde{q}_1 \geq \left(\rho+1\right) \cdot 2^7 \cdot \pi^2$ and $\tilde{q}_{n+1} \geq \tilde{q}^{15}_n \cdot \exp\left(\tilde{q}^7_n \cdot \exp\left(\tilde{q}^6_n\right)\right)$, then there exists a weak mixing Diff$^{\omega}_{\rho}$-diffeomorphism of $\mathbb{T}^2$ that is uniformly rigid with respect to $\left(\tilde{q}_n\right)_{n \in \mathbb{N}}$.
\end{cor}

\section{Definitions and notations}
In this chapter we want to introduce advantageous definitions and notations. In particular, we discuss topologies on the space of diffeomorphisms on $\mathbb{T}^2$. 

\subsection{$C^{\infty}$-topology}

For defining explicit metrics on Diff$^k\left(\mathbb{T}^2\right)$ and in the following the subsequent notations will be useful:
\begin{dfn}
\begin{enumerate}
	\item For a sufficiently differentiable function $f: \mathbb{R}^2 \rightarrow \mathbb{R}$ and a multiindex $\vec{a} = \left(a_1,a_2\right) \in \mathbb{N}^2_0$
\begin{equation*}
D_{\vec{a}}f := \frac{\partial^{\left|\vec{a}\right|}}{\partial x_1^{a_1}\partial x_2^{a_2}} f,
\end{equation*}
where $\left|\vec{a}\right| = a_1 + a_2$ is the order of $\vec{a}$.
  \item For a continuous function $F: \left(0,1\right)^2 \rightarrow \mathbb{R}$
\begin{equation*}
\left\|F\right\|_0 := \sup_{z \in \left(0,1\right)^2} \left|F\left(z\right)\right|.
\end{equation*}
\end{enumerate}
\end{dfn}
For $f,g \in \text{Diff}^k\left(\mathbb{T}^2\right)$ let $F,G: \mathbb{R}^2 \rightarrow \mathbb{R}^2$ denote their lifts. Furthermore, for a function $F: \mathbb{R}^2 \rightarrow \mathbb{R}^2$ we denote by $\left[F\right]_i$ the $i$-th coordinate function.
\begin{dfn}
\begin{enumerate}
	\item For $f,g \in \text{Diff}^k\left(\mathbb{T}^2\right)$ we define
\begin{equation*}
\tilde{d}_0\left(f,g\right) = \max_{i=1,2} \left\{ \inf_{p \in \mathbb{Z}} \left\| \left[F - G\right]_i + p\right\|_0\right\}
\end{equation*}
as well as
\begin{equation*}
\tilde{d}_k\left(f,g\right) = \max \left\{ \tilde{d}_0\left(f,g\right), \left\|D_{\vec{a}}\left[F-G\right]_i\right\|_0 \ : \ i=1,2 \ , \ 1\leq \left|\vec{a}\right| \leq k \right\}.
\end{equation*}
  \item Using the definitions from 1. we define for $f,g \in \text{Diff}^k\left(\mathbb{T}^2\right)$:
  \begin{equation*}
  d_k\left(f,g\right) = \max \left\{ \tilde{d}_k\left(f,g\right) \ , \ \tilde{d}_k\left(f^{-1},g^{-1}\right)\right\}.
  \end{equation*}
\end{enumerate}
\end{dfn}

Obviously $d_k$ describes a metric on Diff$^k\left(\mathbb{T}^2\right)$ measuring the distance between the diffeomorphisms as well as their inverses. As in the case of a general compact manifold the following definition connects to it:

\begin{dfn}
\begin{enumerate}
	\item A sequence of Diff$^{\infty}\left(\mathbb{T}^2\right)$-diffeomorphisms is called convergent in Diff$^{\infty}\left(\mathbb{T}^2\right)$ if it converges in Diff$^k\left(\mathbb{T}^2\right)$ for every $k \in \mathbb{N}$.
	\item On Diff$^{\infty}\left(\mathbb{T}^2\right)$ we declare the following metric
	\begin{equation*}
	d_{\infty}\left(f,g\right) = \sum^{\infty}_{k=1} \frac{d_k\left(f,g\right)}{2^k \cdot \left(1 + d_k\left(f,g\right)\right)}.
	\end{equation*}
\end{enumerate}
\end{dfn}

It is a general fact that Diff$^{\infty}\left(\mathbb{T}^2\right)$ is a complete metric space with respect to this metric $d_{\infty}$. \\
Moreover, we add the adjacent notation:
\begin{dfn}
Let $f \in \text{Diff}^k\left(\mathbb{T}^2\right)$ with lift $F$ be given. Then
\begin{equation*}
\left\| Df \right\|_0 := \max_{i,j \in \left\{1,2\right\}} \left\| D_j \left[F\right]_i \right\|_0
\end{equation*}
and
\begin{equation*}
||| f ||| _k := \max \left\{ \left\|D_{\vec{a}} \left[F\right]_i \right\|_0 , \left\|D_{\vec{a}} \left(\left[F^{-1}\right]_{i}\right)\right\|_0 \ : \ i = 1,2, \ \vec{a} \text{ multiindex with } 0\leq \left| \vec{a}\right| \leq k \right\}.
\end{equation*}
\end{dfn}

\subsection{Analytic topology}
Real-analytic diffeomorphisms of $\mathbb{T}^2$ homotopic to the identity have a lift of type
\begin{equation*}
F\left(\theta,r\right) = \left( \theta + f_1\left(\theta,r\right), r + f_2\left(\theta,r\right)\right),
\end{equation*}
where the functions $f_i: \mathbb{R}^2 \rightarrow \mathbb{R}$ are real-analytic and $\mathbb{Z}^2$-periodic for $i=1,2$. For these functions we introduce the subsequent definition:
\begin{dfn}
For any $\rho > 0$ we consider the set of real-analytic $\mathbb{Z}^2$-periodic functions on $\mathbb{R}^2$, that can be extended to a holomorphic function on $A^{\rho} \coloneqq \left\{ \left(\theta, r\right) \in \mathbb{C}^2 \; : \; \left|\text{im}\theta\right|< \rho, \;\left|\text{im}r\right|< \rho\right\}$. 
\begin{enumerate}
	\item For these functions let $\left\|f\right\|_{\rho} \coloneqq \sup_{\left(\theta,r\right) \in A^{\rho}} \left|f\left(\theta,r\right)\right|$.
	\item The set of these functions satisfying the condition $\left\|f\right\|_{\rho} < \infty$ is denoted by $C^{\omega}_{\rho}\left(\mathbb{T}^2\right)$.
\end{enumerate}
\end{dfn}
Furthermore, we consider the space Diff$^{\omega}_{\rho}\left(\mathbb{T}^2\right)$ of those diffeomorphisms homotopic to the identity, for whose lift we have $f_i \in C^{\omega}_{\rho}\left(\mathbb{T}^2\right)$ for $i=1,2$.
\begin{dfn}
For $f,g \in \text{Diff}^{\omega}_{\rho}\left(\mathbb{T}^2\right)$ we define
\begin{equation*}
\left\|f\right\|_{\rho} = \max_{i=1,2} \left\|f_i\right\|_{\rho}
\end{equation*}
and the distance
\begin{equation*}
d_{\rho}\left(f,g\right) = \max_{i=1,2} \left\{\inf_{p \in \mathbb{Z}} \left\|f_i - g_i - p\right\|_{\rho}\right\}.
\end{equation*}
\end{dfn}
\begin{rem}
Diff$^{\omega}_{\rho}\left(\mathbb{T}^2\right)$ is a Banach space (see \cite{S} or \cite{L} for a more extensive treatment of these spaces).
\end{rem}

Moreover, for a diffeomorphism $T$ with lift $\tilde{T}\left(\theta,r\right)=\left(T_1\left(\theta,r\right), T_2\left(\theta,r\right)\right)$ we define
\begin{equation*}
\left\|DT\right\|_{\rho}= \max\left\{ \left\|\frac{\partial T_1}{\partial \theta}\right\|_{\rho}, \left\|\frac{\partial T_1}{\partial r}\right\|_{\rho}, \left\|\frac{\partial T_2}{\partial \theta}\right\|_{\rho}, \left\|\frac{\partial T_2}{\partial r}\right\|_{\rho} \right\}
\end{equation*}
and use the advantageous notation
\begin{equation*}
\left\|T\right\|_{\rho} = \max\left\{ \inf_{k \in \mathbb{Z}} \sup_{\left(\theta,r\right) \in A_{\rho}} \left|T_1\left(\theta,r\right)-\theta+k\right|, \  \inf_{k \in \mathbb{Z}} \sup_{\left(\theta,r\right) \in A_{\rho}} \left|T_2\left(\theta,r\right)-r+k\right| \right\}.
\end{equation*}

\section{Criterion for weak mixing}
In this section we will formulate a criterion for weak mixing that will be used in the smooth as well as in the real-analytic case. 
\subsection{$\left(\gamma,\delta, \varepsilon\right)$-distribution of horizontal intervals}
Since we work on the manifold $\mathbb{T}^2$, we recall the following definitions stated in \cite{FS}:
\begin{dfn}
Let $\hat{\eta}$ be a partial decomposition of $\mathbb{T}$ into intervals and consider on $\mathbb{T}^2$ the decomposition $\eta$ consisting of intervals in $\hat{\eta}$ times some $r \in \left[0,1\right]$. Sets of this form will be called horizontal intervals and decompositions of this type standard partial decompositions. On the other hand, sets of the form $\left\{\theta\right\} \times J$, where $J$ is an interval on the $r$-axis, are called vertical intervals.
\end{dfn}
Hereby, we can introduce the notion of $\left(\gamma,\delta, \varepsilon\right)$-distribution of a horizontal interval in the vertical direction:
\begin{dfn}
A diffeomorphism $\Phi: \mathbb{T}^2 \rightarrow \mathbb{T}^2$ $\left(\gamma,\delta, \varepsilon\right)$-distributes a horizontal interval $I$ if the following conditions are satisfied
\begin{itemize}
	\item $\pi_r\left(\Phi\left(I\right)\right)$ is an interval $J$ with $1-\delta \leq \lambda\left(J\right)\leq 1$,
	\item $\Phi\left(I\right)$ is contained in a vertical strip $\left[c,c+\gamma\right] \times J$ for some $c \in \mathbb{T}$,
	\item for any interval $\tilde{J}\subseteq J$ we have
	\begin{equation*}
	\left|\frac{\lambda\left(I \cap \Phi^{-1}\left(\mathbb{T} \times \tilde{J}\right)\right)}{\lambda\left(I\right)} - \frac{\lambda\left(\tilde{J}\right)}{\lambda\left(J\right)}\right|\leq \varepsilon \cdot \frac{\lambda\left(\tilde{J}\right)}{\lambda\left(J\right)}.
	\end{equation*}
\end{itemize}
\end{dfn}

\subsection{Statement of the criterion}
The proof of the criterion is the same as in \cite{FS}, section 3. The only difference occurs in comparison to Lemma 3.5., which in our case will be stated in the subsequent way:
\begin{lem}
Let $\left(\eta_n\right)_{n \in \mathbb{N}}$ be a sequence of standard partial decompositions of $\mathbb{T}^2$ into horizontal intervals of length less than $q^{-2.5}_n$. Moreover, let $g_n$ be defined by $g_n\left(\theta,r\right) = \left( \theta + \left[nq^{\sigma}_{n}\right] \cdot r,r\right)$ with some $0< \sigma < 1$ and let $\left(H_n\right)_{n \in \mathbb{N}}$ be a sequence of area-preserving diffeomorphisms such that for every $n \in \mathbb{N}$:
\begin{equation} \tag{C1} \label{eq:C1}
\left\| DH_{n-1} \right\|_0 \leq q^{0.5}_n.
\end{equation}
Consider the partitions $\nu_n \coloneqq \left\{ \Gamma_n = H_{n-1}\left(g_n\left(I_n\right)\right) \;:\; I_n \in \eta_n\right\}$. \\
Then $\eta_n \rightarrow \epsilon$ implies $\nu_n \rightarrow \epsilon$.
\end{lem}
\begin{proof}
For every $\varepsilon >0$ we can choose $n$ large enough such that $\mu\left(\bigcup_{I \in \eta_n} I\right) > 1-\varepsilon$ (because of $\eta_n \rightarrow \epsilon$) and there is a collection of squares $\tilde{S}_n \coloneqq \left\{S_{n,i}\right\}$ with side length between $q^{-1.5}_n$ and $q^{-2}_n$ with total measure of the union $S_n \coloneqq \bigcup_i S_{n,i}$ greater than $1-\sqrt{\varepsilon}$. Then we have $\mu\left(\bigcup_{I \in \eta_n} I \cap S_n\right) \geq \left(1-\sqrt{\varepsilon}\right) \cdot \mu\left(S_{n}\right)$, because otherwise $\mu\left(S_n \setminus \bigcup_{I \in \eta_n} I\right) > \sqrt{\varepsilon} \cdot \mu\left(S_n\right) > \sqrt{\varepsilon} \cdot \left(1-\sqrt{\varepsilon}\right)$ and so $\mu\left(\mathbb{T}^2\setminus \bigcup_{I \in \eta_n} I\right) > \sqrt{\varepsilon} - \varepsilon > \varepsilon$ in case of $\varepsilon < \frac{1}{4}$, which contradicts $\mu\left(\bigcup_{I \in \eta_n} I\right) > 1-\varepsilon$. Since the horizontal intervals $I \in \eta_n$ have length less than $q^{-2.5}_n$, we can approximate the squares in the above collection $\tilde{S}_n$ for $n$ sufficiently large in such a way that $\mu\left(\bigcup_{I \in \eta_n, I \subset S_{n}} I\right) \geq \left(1-2\sqrt{\varepsilon}\right) \cdot \mu\left(S_{n}\right)$. \\
In the next step we consider the sets $C_{n,i} \coloneqq H_{n-1}\left(g_n\left(S_{n,i}\right)\right)$ with $S_{n,i} \in \tilde{S}_n$. For these sets $C_{n,i}$ we have:
\begin{equation*}
\text{diam}\left(C_{n,i}\right) \leq \left\|DH_{n-1}\right\|_0 \cdot \left\|Dg_n\right\|_0 \cdot \text{diam}\left(S_{n,i}\right) \leq q^{0.5}_n \cdot n \cdot q^{\sigma}_n \cdot \sqrt{2} \cdot q^{-1.5}_n = n \cdot \sqrt{2} \cdot q^{\sigma-1}_n,
\end{equation*}
which goes to $0$ as $n \rightarrow \infty$ because $\sigma <1$. Therefore, any Borel set $B$ can be approximated by a union of such sets $C_{n,i}$ with any prescribed accuracy if $n$ is sufficiently large, i.e. for every $\varepsilon > 0$ there is $N \in \mathbb{N}$ such that for $n \geq N$ there is an index set $J_n$: $\mu\left( B \triangle \bigcup_{i \in J_n} C_{n,i}\right)< \varepsilon$. Now we choose the union of these elements $I \in \eta_n$ contained in the occurring cubes $S_{n,i}$ and obtain: $\mu\left(B \triangle \bigcup H_{n-1} \circ g_n\left(I\right)\right) \leq \mu\left( B \triangle \bigcup_{i \in J_n} C_{n,i}\right) + \mu\left(S_n \setminus \bigcup_{I \in \eta_n, I \subset S_{n}} I\right) < \varepsilon + 2 \sqrt{\varepsilon} \cdot \mu\left(S_n\right) < 3 \sqrt{\varepsilon}$. Thus, $B$ gets well approximated by unions of elements of $\nu_n$ if $n$ is chosen sufficiently large.
\end{proof}
Now the criterion for weak mixing can be stated in the following way (compare with \cite{FS}, Proposition 3.9.):
\begin{prop} \label{prop:wm}
Let $f_n = H_n \circ R_{\alpha_{n+1}} \circ H^{-1}_n$ be diffeomorphisms constructed as explained in the introduction with $0< \sigma < \frac{1}{2}$ and such that $\left\| DH_{n-1} \right\|_0 \leq q^{0.5}_n$ holds for all $n \in \mathbb{N}$. \\
Suppose that the limit $f \coloneqq \lim_{n \rightarrow \infty} f_n$ exists. If there exists a sequence $\left(m_n\right)_{n \in \mathbb{N}}$ of natural numbers satisfying $d_0\left(f^{m_n}_n, f^{m_n}\right)< \frac{1}{2^n}$ and a sequence $\left(\eta_n\right)_{n \in \mathbb{N}}$ of standard partial decompositions of $\mathbb{T}^2$ into horizontal intervals of length less than $q^{-2.5}_n$ such that $\eta_n \rightarrow \epsilon$ and the diffeomorphism $\Phi_n \coloneqq \phi_n \circ R^{m_n}_{\alpha_{n+1}} \circ \phi^{-1}_{n}$ $\left(\frac{1}{nq^{\sigma}_{n}}, \frac{1}{n}, \frac{1}{n}\right)$-distributes every interval $I_n \in \eta_n$, then the limit diffeomorphism $f$ is weak mixing.
\end{prop}

\begin{rem}
In \cite{FS} it is demanded $\left\|DH_{n-1}\right\|_0 < \ln\left(q_n\right)$ instead of requirement \ref{eq:C1}. We did this modification because the fulfilment of the original condition would lead to stricter requirements on the rigidity sequence: In particular, equation \ref{eq:C1diff} would require an exponential growth rate.
\end{rem}

\section{Convergence of $\left(f_n\right)_{n \in \mathbb{N}}$ in Diff$^{\infty}\left(\mathbb{T}^2\right)$}

\subsection{Properties of the conjugation maps $h_n$ and $H_n$}
Using the explicit definitions of the maps $g_n$, $\phi_n$ we can compute
\begin{equation*}
h_n\left(\theta,r\right) = \left(\theta + \left[nq^{\sigma}_n\right]\cdot r + \left[nq^{\sigma}_n\right] \cdot q^2_n \cdot \cos\left(2 \pi q_n \theta\right), r + q^2_n \cdot \cos\left(2 \pi q_n \theta\right)\right)
\end{equation*}
as well as
\begin{equation*}
h^{-1}_n\left(\theta,r\right) = \left(\theta - \left[nq^{\sigma}_n\right]\cdot r, r - q^2_n \cdot \cos\left(2 \pi q_n\left(\theta - \left[nq^{\sigma}_n\right]\cdot r\right)\right)\right).
\end{equation*}
Then we can estimate for every $k \in \mathbb{N}$ and every multiindex $\vec{a} \in \mathbb{N}^2_0$, $\left|\vec{a}\right|\leq k$:
\begin{equation*}
\left\|D_{\vec{a}} h_n \right\|_0 \leq 2 \cdot \left(2 \pi\right)^k \cdot \left[nq^{\sigma}_n\right] \cdot q^{2+k}_n
\end{equation*}
and
\begin{equation*}
\left\|D_{\vec{a}} h^{-1}_n \right\|_0 \leq 2 \cdot \left(2 \pi\right)^k \cdot \left[nq^{\sigma}_n\right]^k \cdot q^{2+k}_n.
\end{equation*}
Thus, we obtain
\begin{equation} \label{eq:h_n}
|||h_n|||_k \leq 2^{k+1} \cdot \pi^k \cdot q^{2+k}_n \cdot n^k \cdot q^{\sigma \cdot k}_n \leq \left(2 \pi n q^2_n\right)^{k+1}.
\end{equation}
In the next step we want to deduce norm estimates for the conjugation map $H_n = H_{n-1} \circ h_n$. Therefore, we have to understand the derivatives of a composition of maps: 
\begin{lem} \label{lem:comp}
Let $g,h \in \text{Diff}^{\infty}\left(\mathbb{T}^2\right)$ and $k \in \mathbb{N}$. Then for the composition $g \circ h$ it holds
\begin{equation*}
|||g \circ h |||_k \leq \left(k+1\right)! \cdot |||g|||^k_k \cdot |||h|||^k_k.
\end{equation*}
\end{lem}
\begin{proof}
By induction on $k \in \mathbb{N}$ we will prove the following observation: \\
\textbf{Claim: }\textit{For any multiindex $\vec{a} \in \mathbb{N}^2_0$ with $\left|\vec{a}\right|=k$ and $i\in \left\{1,2\right\}$ the partial derivative $D_{\vec{a}} \left[g \circ h\right]_i$ consists of at most $\left(k+1\right)!$ summands, where each summand is the product of one derivative of $g$ of order at most $k$ and at most $k$ derivatives of $h$ of order at most $k$.}
\begin{itemize}
	\item \textit{Start:} $k=1$ \\
	For $i_1, i \in \left\{1,2\right\}$ we compute:
	\begin{equation*}
	D_{x_{i_1}} \left[g \circ h\right]_i\left(x_1,x_2\right) = \sum^{2}_{j_1=1} \left(D_{x_{j_1}} \left[g\right]_i\right)\left(h\left(x_1,x_2\right)\right) \cdot D_{x_{i_1}} \left[h\right]_{j_1}\left(x_1,x_2\right).
	\end{equation*}
	Hence, this derivative consists of $2!=2$ summands and each summand has the announced form.
	\item \textit{Induction assumption:} The claim holds for $k \in \mathbb{N}$.
	\item \textit{Induction step:} $k\rightarrow k+1$ \\
	Let $i \in \left\{1,2\right\}$ and $\vec{b} \in \mathbb{N}^2_0$ be any multiindex of order $\left|\vec{b}\right|=k+1$. There are $j \in \left\{1,2\right\}$ and a multiindex $\vec{a}$ of order $\left|\vec{a}\right|=k$ such that $D_{\vec{b}}= D_{x_j}D_{\vec{a}}$. By the induction assumption the partial derivative $D_{\vec{a}}\left[g \circ h\right]_i$ consists of at most $\left(k+1\right)!$ summands, at which the summand with the most factors is of the subsequent form:
	\begin{equation*}
	D_{\vec{c}_1} \left[g\right]_i\left(h\left(x_1,x_2\right)\right) \cdot D_{\vec{c}_2}\left[h\right]_{i_2}\left(x_1,x_2\right) \cdot ... \cdot D_{\vec{c}_{k+1}}\left[h\right]_{i_{k+1}}\left(x_1,x_2\right),
	\end{equation*}
	where each $\vec{c}_i$ is of order at most $k$. Using the product rule we compute how the derivative $D_{x_j}$ acts on such a summand: 
	\begin{align*}
	&\left(\sum^{2}_{j_1=1}D_{x_{j_1}}D_{\vec{c}_1} \left[g\right]_i\circ h \cdot D_{x_j}\left[h\right]_{j_1} D_{\vec{c}_2}\left[h\right]_{i_2} \cdot ... \cdot D_{\vec{c}_{k+1}}\left[h\right]_{i_{k+1}} \right) +  \\
	&D_{\vec{c}_1} \left[g\right]_i\circ h \cdot D_{x_j}D_{\vec{c}_2}\left[h\right]_{i_2}\cdot ... \cdot D_{\vec{c}_{k+1}}\left[h\right]_{i_{k+1}}+...+D_{\vec{c}_1} \left[g\right]_i\circ h \cdot D_{\vec{c}_2}\left[h\right]_{i_2} \cdot ... \cdot D_{x_j}D_{\vec{c}_{k+1}}\left[h\right]_{i_{k+1}}
	\end{align*}
	Thus, each summand is the product of one derivative of $g$ of order at most $k+1$ and at most $k+1$ derivatives of $h$ of order at most $k+1$. Moreover, we observe that $2 + k$ summands arise out of one. So the number of summands can be estimated by $\left(k+2\right) \cdot \left(k+1\right)! = \left(k+2\right)!$ and the claim is verified.
\end{itemize}
Using this claim we obtain for $i \in \left\{1,2\right\}$ and any multiindex $\vec{a} \in \mathbb{N}^2_0$ of order $\left|\vec{a}\right|=k$:
\begin{equation*}
\left\|D_{\vec{a}} \left[g \circ h\right]_i\right\|_0 \leq \left(k+1\right)! \cdot |||g|||_k \cdot |||h|||^k_k.
\end{equation*}
Applying the claim on $h^{-1} \circ g^{-1}$ yields:
\begin{equation*}
\left\|D_{\vec{a}} \left[h^{-1} \circ g^{-1}\right]_i\right\|_0 \leq \left(k+1\right)! \cdot |||g|||^k_k \cdot |||h|||_k.
\end{equation*}
We conclude
\begin{equation*}
|||g \circ h |||_k \leq \left(k+1\right)! \cdot |||g|||^k_k \cdot |||h|||^k_k.
\end{equation*}
\end{proof}
 
Using this result we compute for every $k \in \mathbb{N}$: 
\begin{equation} \label{eq:H_n}
|||H_n|||_{k} \leq \left(k+1\right)! \cdot |||H_{n-1}|||^k_k \cdot |||h_n|||^k_k.
\end{equation}
Hereby, we can deduce the subsequent estimate of the norm $|||H_n|||_{k+1}$ under some assumptions on the growth rate of the numbers $q_n$:
\begin{lem} \label{lem:H}
Let $k,n \in \mathbb{N}$ and $n\geq2$. Assume
\begin{equation} \tag{A2} \label{eq:billig}
q_{n+1} \geq 2 \cdot \pi \cdot n \cdot q^2_n.
\end{equation}
Then we have
\begin{equation*}
|||H_n|||_{k+1} \leq \left(\left(k+2\right)!\right)^{\left(k+2\right)^{n-2}} \cdot \left(2 \pi n q_n\right)^{\left(k+2\right) \cdot \left(k+1\right)^{n-1} \cdot \left(n+1\right)}.
\end{equation*}
\end{lem}

\begin{proof}
Let $k \in \mathbb{N}$ be arbitrary. We proof this result by induction on $n$:
\begin{itemize}
	\item \textit{Start:} $n=2$ \\
	Using equation \ref{eq:H_n} and the norm estimate on $h_n$ from equation \ref{eq:h_n} we obtain the claim:
	\begin{align*}
	|||H_2|||_{k+1} & \leq \left(k+2\right)! \cdot |||H_{1}|||^{k+1}_{k+1} \cdot |||h_2|||^{k+1}_{k+1} \\
	& = \left(k+2\right)! \cdot |||h_{1}|||^{k+1}_{k+1} \cdot |||h_2|||^{k+1}_{k+1} \\
	& \leq \left(k+2\right)! \cdot \left(\left(2 \pi q^2_1\right)^{k+2}\right)^{k+1} \cdot \left(\left(2 \pi \cdot 2 \cdot q^2_2\right)^{k+2}\right)^{k+1} \\
	& \leq \left(k+2\right)! \cdot q^{\left(k+2\right)\cdot \left(k+1\right)}_2 \cdot \left(2 \pi \cdot 2 \cdot q_2\right)^{2 \cdot \left(k+2\right) \cdot \left(k+1\right)} \\
	& \leq \left(k+2\right)! \cdot \left(2 \pi \cdot 2 \cdot q_2\right)^{3 \cdot \left(k+2\right) \cdot \left(k+1\right)} \\
	& = \left(\left(k+2\right)!\right)^{\left(k+2\right)^{2-2}} \cdot \left(2 \pi \cdot 2 \cdot q_2\right)^{\left(k+2\right) \cdot \left(k+1\right)^{2-1} \cdot \left(2+1\right)}
	\end{align*}
	\item \textit{Induction assumption:} The claim is true for $n \in \mathbb{N}$, $n\geq2$.
	\item \textit{Induction step $n \rightarrow n+1$:} \\
	Using equation \ref{eq:H_n}, the norm estimate on $h_n$ from equation \ref{eq:h_n} and the induction assumption we compute:
	\begin{align*}
	& |||H_{n+1}|||_{k+1} \leq \left(k+2\right)! \cdot |||H_{n}|||^{k+1}_{k+1} \cdot |||h_{n+1}|||^{k+1}_{k+1} \\
	& \leq \left(k+2\right)! \cdot \left(\left(\left(k+2\right)!\right)^{\left(k+2\right)^{n-2}} \cdot \left(2 \pi n q_n\right)^{\left(k+2\right) \cdot \left(k+1\right)^{n-1} \cdot \left(n+1\right)}\right)^{k+1} \cdot \left(\left(2 \pi \left(n+1\right) q^2_{n+1} \right)^{k+2}\right)^{k+1} \\
	& \leq \left(k+2\right)! \cdot \left(\left(k+2\right)!\right)^{\left(k+2\right)^{n-2} \cdot \left(k+1\right)} \cdot q^{\left(k+2\right) \cdot \left(k+1\right)^n \cdot \left(n+1\right)}_{n+1} \cdot \left(2 \pi \left(n+1\right)q_{n+1}\right)^{2 \cdot \left(k+2\right)\cdot \left(k+1\right)} \\
	& \leq \left(\left(k+2\right)!\right)^{\left(k+2\right)^{n-1}} \cdot \left(2 \pi \left(n+1\right)q_{n+1}\right)^{\left(k+2\right) \cdot \left(k+1\right)^n \cdot \left(n+2\right)},
	\end{align*}
	where we used in the last step the subsequent estimation:
\begin{align*}
& \left(k+2\right) \cdot \left(k+1\right)^n \cdot \left(n+1\right) + 2 \cdot \left(k+2\right) \cdot \left(k+1\right) = \left(k+2\right) \cdot \left(k+1\right) \cdot \left(\left(k+1\right)^{n-1} \cdot \left(n+1\right)+2\right) \\
& \leq \left(k+2\right) \cdot \left(k+1\right) \cdot \left(k+1\right)^{n-1} \cdot \left(n+2\right) = \left(k+2\right) \cdot \left(k+1\right)^{n} \cdot \left(n+2\right).
\end{align*}
\end{itemize}
\end{proof}
\begin{rem}
As a special case of Lemma \ref{lem:comp} we have $\left\|DH_n\right\|_0 \leq 2! \cdot \left\|DH_{n-1}\right\|_0 \cdot \left\|Dh_n\right\|_0$. With the aid of equation \ref{eq:h_n} we can estimate:
\begin{equation*}
\left\|DH_n\right\|_0 \leq 2! \cdot q^{0.5}_n \cdot \left(2 \pi \cdot n \cdot q^2_n\right)^2 = 8 \pi^2 \cdot n^2 \cdot q^{4.5}_n,
\end{equation*}
where we used condition \ref{eq:C1}, i.e. $\left\|DH_{n-1}\right\|_0 \leq q^{0.5}_n$. In order to guarantee this property for $DH_n$ we demand:
\begin{equation} \tag{A3} \label{eq:C1diff}
q_{n+1} \geq \left\|DH_n\right\|^2_0 \geq \left(8 \pi^2 \cdot n^2 \cdot q^{4.5}_n\right)^2 = 64 \pi^4 \cdot n^4 \cdot q^9_n.
\end{equation}
\end{rem}

\subsection{Proof of Convergence} \label{subsection:convA}
In the proof of convergence the following result, which is more precise than \cite{FS}, Lemma 5.6., is useful:
\begin{lem} \label{lem:conj}
Let $k \in \mathbb{N}_0$ and $h \in \text{Diff}^{\infty}\left(\mathbb{T}^2\right)$. Then for all $\alpha, \beta \in \mathbb{R}$ we obtain:
\begin{equation*}
d_k\left(h \circ R_{\alpha} \circ h^{-1}, h \circ R_{\beta} \circ h^{-1}\right) \leq C_k \cdot |||h|||^{k+1}_{k+1} \cdot \left|\alpha - \beta\right|,
\end{equation*}
where $C_k = \left(k+1\right)!$.
\end{lem}
\begin{proof}
As an application of the claim in the proof of Lemma \ref{lem:comp} we observe \\
\textbf{Fact: }\textit{For any $\vec{a} \in \mathbb{N}^2_0$ with $\left|\vec{a}\right|=k$ and $i\in \left\{1,2\right\}$ the partial derivative $D_{\vec{a}} \left[h \circ R_{\alpha} \circ h^{-1}\right]_i$ consists of at most $\left(k+1\right)!$ summands, where each summand is the product of one derivative of $h$ of order at most $k$ and at most $k$ derivatives of $h^{-1}$ of order at most $k$.} \\
Furthermore, with the aid of the mean value theorem we can estimate for any multiindex $\vec{a} \in \mathbb{N}^2_0$ with $\left|\vec{a}\right|\leq k$ and $i\in \left\{1,2\right\}$:
\begin{align*}
\left|D_{\vec{a}}\left[h\right]_i\left(R_{\alpha} \circ h^{-1}\left(x_1,x_2\right)\right)-D_{\vec{a}}\left[h\right]_i\left(R_{\beta} \circ h^{-1}\left(x_1,x_2\right)\right)\right| \leq |||h|||_{k+1} \cdot \left|\alpha-\beta\right|.
\end{align*}
Since $\left(h_n \circ R_{\alpha} \circ h^{-1}_n\right)^{-1}= h_n \circ R_{-\alpha} \circ h^{-1}_n$ is of the same form, we obtain in conclusion:
\begin{align*}
d_k\left(h \circ R_{\alpha} \circ h^{-1}, h \circ R_{\beta} \circ h^{-1}\right) & \leq \left(k+1\right)! \cdot |||h|||_{k+1} \cdot |||h|||^k_{k} \cdot \left| \alpha - \beta \right| \\
& \leq \left(k+1\right)! \cdot |||h|||^{k+1}_{k+1} \cdot \left| \alpha - \beta \right|.
\end{align*}
\end{proof}
Under some conditions on the proximity of $\alpha_n$ and $\alpha_{n+1}$ we can prove convergence:
\begin{lem} \label{lem:conv}
We assume 
\begin{equation} \tag{A1} \label{eq:alpha}
\left|\alpha_{n+1} - \alpha_n \right| \leq \frac{1}{2^n \cdot \left(n+1\right)! \cdot q_n \cdot |||H_n|||^{n + 1}_{n +1}}.
\end{equation}
Then the diffeomorphisms $f_n = H_n \circ R_{\alpha_{n+1}} \circ H^{-1}_n$ satisfy:
\begin{itemize}
	\item The sequence $\left(f_n\right)_{n \in \mathbb{N}}$ converges in the Diff$^{\infty}\left(\mathbb{T}^2\right)$-topology to a measure-preserving diffeomorphism $f$.
	\item We have for every $n \in \mathbb{N}$ and $m \leq q_{n+1}$:
	\begin{equation*}
	d_0\left(f^{m}, f^{m}_n\right) < \frac{1}{2^n}.
	\end{equation*}
\end{itemize}
\end{lem}

\begin{proof}
\begin{enumerate}
	\item According to our construction it holds $h_n \circ R_{\alpha_n} = R_{\alpha_n} \circ h_n$ and hence we can apply Lemma \ref{lem:conj} for every $k,n \in \mathbb{N}$:
	\begin{equation*}
	d_k\left(f_n, f_{n-1}\right) = d_k\left(H_n \circ R_{\alpha_{n+1}} \circ H^{-1}_{n}, H_n \circ R_{\alpha_n} \circ H^{-1}_{n}\right) \leq C_k \cdot ||| H_n |||^{k+1}_{k+1} \cdot \left| \alpha_{n+1} - \alpha_n \right|.
	\end{equation*}
	By the assumptions of this Lemma it follows for every $k \leq n$:
	\begin{equation} \label{est4}
	d_k\left(f_n,f_{n-1}\right) \leq d_{n}\left(f_n,f_{n-1}\right) \leq C_{n} \cdot ||| H_n |||^{n+1}_{n +1} \cdot \frac{1}{2^n \cdot C_{n} \cdot q_n \cdot ||| H_n|||^{n+1}_{n +1} } < \frac{1}{2^n}.
	\end{equation}
	In the next step we show that for arbitrary $k \in \mathbb{N}$ $\left(f_n\right)_{n \in \mathbb{N}}$ is a Cauchy sequence in Diff$^k\left(\mathbb{T}^2\right)$, i.e. $\lim_{n,m\rightarrow \infty} d_k\left(f_n,f_m\right) = 0$. For this purpose, we calculate:
	\begin{equation} \label{est2}
	\lim_{n \rightarrow \infty} d_k\left(f_n,f_m\right) \leq \lim_{n \rightarrow \infty} \sum^{n}_{i=m+1} d_k\left(f_i, f_{i-1}\right) = \sum^{\infty}_{i=m+1}d_k\left(f_i, f_{i-1}\right).
	\end{equation}
	We consider the limit process $ m \rightarrow \infty$, i.e. we can assume $k\leq m$ and obtain from equations \ref{est4} and \ref{est2}:
	\begin{equation*}
	\lim_{n,m \rightarrow \infty} d_k\left(f_n,f_m\right) \leq \lim_{m \rightarrow \infty} \sum^{\infty}_{i=m+1} \frac{1}{2^i} = 0.
	\end{equation*}
	Since Diff$^{k}\left(\mathbb{T}^2\right)$ is complete, the sequence $\left(f_n\right)_{n \in \mathbb{N}}$ converges consequently in Diff$^k\left(\mathbb{T}^2\right)$ for every $k \in \mathbb{N}$. Thus, the sequence converges in Diff$^{\infty}\left(\mathbb{T}^2\right)$ by definition. 
	\item Again with the help of Lemma \ref{lem:conj} we compute for every $i \in \mathbb{N}$:
\begin{equation*}
d_0\left(f^{m}_{i},f^{m}_{i-1}\right) = d_0\left(H_{i} \circ R_{m \cdot \alpha_{i+1}} \circ H^{-1}_{i},  H_{i} \circ R_{m \cdot \alpha_{i}} \circ H^{-1}_{i}\right) \leq ||| H_{i} |||_1 \cdot m  \cdot \left|\alpha_{i+1}-\alpha_{i}\right|.
\end{equation*}
Since $m \leq q_{n+1} \leq q_i$ we conclude for every $i > n$:
\begin{equation*}
d_0 \left( f^{m}_{i}, f^{m}_{i-1}\right) \leq ||| H_i |||_1 \cdot m \cdot \frac{1}{2^i \cdot \left(i+1\right)! \cdot q_i \cdot ||| H_i |||^{i +1}_{i +1}} < \frac{m}{q_i} \cdot \frac{1}{2^i} \leq \frac{1}{2^i}.
\end{equation*}
Thus, for every $m\leq q_{n+1}$ we get the claimed result:
\begin{equation*}
d_0\left(f^{m}, f^{m}_{n}\right) = \lim_{k \rightarrow \infty} d_0\left(f^{m}_{k}, f^{m}_{n}\right) \leq \lim_{k \rightarrow \infty} \sum^{k}_{i=n+1} d_0\left(f^{m}_{i}, f^{m}_{i-1}\right) < \sum^{\infty}_{i=n+1} \frac{1}{2^{i}} = \left(\frac{1}{2}\right)^n.
\end{equation*}
\end{enumerate}
\end{proof}
\begin{rem} \label{rem:rigid}
By definition $\tilde{q}_{n+1} \leq q_{n+1}$. Hence, the second statement of the previous Lemma implies $d_0\left(f^{\tilde{q}_{n+1}}_n, f^{\tilde{q}_{n+1}}\right)<\frac{1}{2^n}$. Since the number $\alpha_{n+1}$ was chosen in such a way that $f^{\tilde{q}_{n+1}}_n = id$, we have $d_0\left(id, f^{\tilde{q}_{n+1}}\right)<\frac{1}{2^n}$, which goes to zero as $n\rightarrow \infty$. Thus, $\left(\tilde{q}_{n}\right)_{n \in \mathbb{N}}$ is an uniform rigidity sequence of $f$.
\end{rem}
By Lemma \ref{lem:H} we can satisfy the requirement \ref{eq:alpha} if we demand:
\begin{equation*}
\left|\alpha_{n+1}-\alpha_n\right|\leq\frac{1}{2^n \cdot  \left(n+1\right)! \cdot q_n \cdot \left(\left(n + 2\right)!\right)^{\left(n+2\right)^{n-2} \cdot \left(n +1\right)} \cdot \left(2 \pi n q_n\right)^{\left(n+2\right) \cdot \left(n+1\right)^{n+1}}  }.
\end{equation*}
Since $\left|\alpha_{n+1}-\alpha_n\right| = \frac{a_n}{q_n \cdot \tilde{q}_{n+1}}\leq \frac{1}{\tilde{q}_{n+1}}$ this requirement can be met if we demand
\begin{equation*}
\tilde{q}_{n+1} \geq  \varphi_1\left(n\right) \cdot q^{\left(n+2\right) \cdot \left(n+1\right)^{n+1}+1}_n,
\end{equation*}
at which $\varphi_1\left(n\right) \coloneqq 2^{n} \cdot \left(n+1\right)! \cdot \left(\left(n + 2\right)!\right)^{\left(n+2\right)^{n-2} \cdot \left(n +1\right)} \cdot \left(2 \pi n\right)^{\left(n+2\right) \cdot \left(n+1\right)^{n+1}}$. Hereby, the other conditions \ref{eq:C1diff} and \ref{eq:billig} are fulfilled. \\
Using $q_n = q_{n-1} \cdot  \tilde{q}_n < \tilde{q}^2_n$ this yields the condition
\begin{equation*}
\tilde{q}_{n+1} \geq \varphi_1\left(n\right) \cdot  \tilde{q}^{2 \cdot \left(\left(n+2\right) \cdot \left(n+1\right)^{n+1} +1\right)}_n.
\end{equation*}

\section{Convergence of $\left(f_n\right)_{n \in \mathbb{N}}$ in Diff$^{\omega}_{\rho}\left(\mathbb{T}^2\right)$} \label{sec:ana}
Let $\rho>0$ be given.
\subsection{Properties of the conjugation maps $h_n$ and $H_n$}
Regarded as a function $h_n: \mathbb{C}^2 \rightarrow \mathbb{C}^2$ we have
\begin{align*}
& h^{-1}_n\left(\theta,r\right) = h^{-1}_n\left(\theta_1 + \imath \cdot \theta_2,r_1 + \imath \cdot r_2\right) \\
= & \left(\theta_1 + \imath \cdot \theta_2- \left[nq^{\sigma}_n\right]\cdot \left(r_1 + \imath \cdot r_2\right), r_1 + \imath \cdot r_2 - q^2_n \cdot \cos\left(2 \pi q_n\left(\theta_1 + \imath \cdot \theta_2 - \left[nq^{\sigma}_n\right]\cdot \left(r_1 + \imath \cdot r_2\right)\right)\right)\right).
\end{align*}
Since for $\left(\theta_1 + \imath \cdot \theta_2, r_1 + \imath \cdot r_2\right) \in A_{\rho}$ it holds $- \rho < r_2 < \rho$ and $- \rho < \theta_2 < \rho$, we can estimate:
\begin{align*}
\inf_{k \in \mathbb{Z}} \sup_{\left(\theta,r\right) \in A_{\rho}} \left|\left[h^{-1}_n\right]_1 - \theta + k \right| & = \inf_{k \in \mathbb{Z}} \sqrt{\left( - \left[nq^{\sigma}_n\right]\cdot r_1+k\right)^2 + \left( - \left[nq^{\sigma}_n\right]\cdot r_2\right)^2} \\
&\leq \left[nq^{\sigma}_n\right] \cdot \sqrt{1+\rho^2} 
\end{align*}
and
\begin{align*}
\inf_{k \in \mathbb{Z}} \sup_{\left(\theta,r\right) \in A_{\rho}} \left|\left[h^{-1}_n\right]_2- r + k \right| \leq 2 \cdot q^2_n \cdot \exp\left( 2 \pi q_n \cdot \rho + 2 \pi q_n \cdot \left[nq^{\sigma}_n\right] \cdot \rho\right).
\end{align*}
Hence, it holds (note that we demand $q_1 \geq \tilde{\rho}_0 = \rho +1$ in equation \ref{eq:rho}):
\begin{equation} \label{eq:rho_n}
\left\|h^{-1}_n\right\|_{\rho} \leq 2 \cdot q^2_n \cdot \exp\left( 2 \pi \cdot q_n \cdot \rho \cdot \left(1+ \left[nq^{\sigma}_n\right] \right)\right).
\end{equation}
We introduce the subsequent quantities:
\begin{align*}
& \rho_n \coloneqq \left\|H^{-1}_n\right\|_{\rho}, \ \  & \rho_0 \coloneqq \rho \\
& \tilde{\rho}_n \coloneqq 2 \cdot q^2_n \cdot \exp\left( 2 \pi \cdot q_n \cdot \tilde{\rho}_{n-1} \cdot \left(1+ \left[nq^{\sigma}_n\right] \right)\right),\ \  & \tilde{\rho}_0 \coloneqq \rho+1 \end{align*}
Using equation \ref{eq:rho_n} we obtain
\begin{equation*}
\rho_n = \left\|h^{-1}_n \circ H^{-1}_{n-1}\right\|_{\rho} \leq \left\|h^{-1}_n\right\|_{\rho_{n-1}} \leq 2 \cdot q^2_n \cdot \exp\left(  2 \pi q_n \cdot \rho_{n-1} \cdot \left(1+\left[nq^{\sigma}_n\right]\right)\right).
\end{equation*}
We state the following relation between the quantities:
\begin{rem} \label{rem:plus}
We have $\rho_n +1 \leq \tilde{\rho}_n$ for every $n \in \mathbb{N}$.
\end{rem}
\begin{proof}
We prove this result by induction on $n$:
\begin{itemize}
	\item \textit{Start:} $n=1$ \\
	By the above formula we verify:
	\begin{align*}
	\tilde{\rho}_1 & = 2 \cdot q^2_1 \cdot \exp\left( 2 \pi \cdot q_1 \cdot \tilde{\rho}_0 \cdot \left(1+ \left[q^{\sigma}_1\right] \right)\right) \\
	& = 2 \cdot q^2_1 \cdot \exp\left( 2 \pi \cdot q_1 \cdot \rho \cdot \left(1+ \left[q^{\sigma}_1\right] \right)\right) \cdot \exp\left( 2 \pi \cdot q_1 \cdot \left(1+ \left[q^{\sigma}_1\right]\right)\right) \\
	& \geq 2 \cdot q^2_1 \cdot \exp\left( 2 \pi \cdot q_1 \cdot \rho \cdot \left(1+ \left[q^{\sigma}_1\right] \right)\right) + 1 \geq \rho_1 + 1. 
	\end{align*}
	\item \textit{Induction assumption:} The claim holds for $n-1$.
	\item \textit{Induction step:} $n-1 \rightarrow n$ \\
	Using the induction assumption we can compute in the same way:
	\begin{align*}
	\tilde{\rho}_n & = 2 \cdot q^2_n \cdot \exp\left( 2 \pi \cdot q_n \cdot \tilde{\rho}_{n-1} \cdot \left(1+ \left[nq^{\sigma}_n\right] \right)\right) \\
	& \geq 2 \cdot q^2_n \cdot \exp\left( 2 \pi \cdot q_n \cdot \rho_{n-1} \cdot \left(1+ \left[nq^{\sigma}_n\right] \right)\right) \cdot \exp\left( 2 \pi \cdot q_n \cdot \left(1+ \left[nq^{\sigma}_n\right]\right)\right) \\
	& \geq 2 \cdot q^2_n \cdot \exp\left( 2 \pi \cdot q_n \cdot \rho_{n-1} \cdot \left(1+ \left[nq^{\sigma}_n\right] \right)\right) + 1 \geq \rho_n + 1.
	\end{align*}
\end{itemize}
\end{proof}
We demand
\begin{equation} \tag{B2'} \label{eq:rho}
q_{n+1} \geq \tilde{\rho}_n. 
\end{equation}
This yields the condition: $q_{n+1} \geq 2 \cdot q^2_n \cdot \exp\left(2 \pi q_n \cdot \tilde{\rho}_{n-1} \cdot \left(1+\left[nq^{\sigma}_n\right]\right)\right)$. And since we demand $\tilde{\rho}_{n-1} \leq q_n$ we require for the sequence $\left(q_n\right)_{n \in \mathbb{N}}$:
\begin{equation} \tag{B2} \label{eq:rhoc}
q_{n+1} \geq 2 \cdot q^2_n \cdot \exp\left(2 \pi \cdot q^2_n \cdot \left(1+n q^{\sigma}_n\right)\right).
\end{equation}

Furthermore, recall that
\begin{equation*}
h_n\left(\theta,r\right) = \left(\theta + \left[nq^{\sigma}_n\right]\cdot r + \left[nq^{\sigma}_n\right] \cdot q^2_n \cdot \cos\left(2 \pi q_n \theta\right), r + q^2_n \cdot \cos\left(2 \pi q_n \theta\right)\right).
\end{equation*}
The occurring partial derivatives are
\begin{align*}
& \frac{\partial \left[h_n\right]_1}{\partial \theta} = 1-\left[nq^{\sigma}_n\right] \cdot 2 \pi \cdot q^3_n \cdot \sin\left(2 \pi q_n \theta\right) \ & \frac{\partial \left[h_n\right]_1}{\partial r} = \left[nq^{\sigma}_n\right] \\
& \frac{\partial \left[h_n\right]_2}{\partial \theta} = -2 \pi \cdot q^3_n \cdot \sin\left(2 \pi q_n \theta\right) \ & \frac{\partial \left[h_n\right]_2}{\partial r} =1
\end{align*}
Thus, in order to calculate $\left\|Dh_n\right\|_{\rho}$, we have to examine $\left\|\frac{\partial \left[h_n\right]_1}{\partial \theta}\right\|_{\rho}$:
\begin{equation*}
\left\|Dh_n\right\|_{\rho} = \left\|\frac{\partial \left[h_n\right]_1}{\partial \theta}\right\|_{\rho} \leq 1 + \left[nq^{\sigma}_n\right] \cdot 2 \pi \cdot q^3_n \cdot \exp\left(2 \pi \cdot q_n \cdot \rho\right) \leq  4 \pi \cdot n \cdot q^{3+\sigma}_n \cdot \exp\left(2 \pi \cdot q_n \cdot \rho\right).
\end{equation*}
Under condition \ref{eq:rho} we can estimate with the aid of Remark \ref{rem:plus}:
\begin{align} \label{align:Dh}
\left\|Dh_n\right\|_{\rho_n +1} & \leq 4 \pi n \cdot q^{3+\sigma}_n \cdot \exp\left(2 \pi \cdot q_n \cdot \left(\rho_n+1\right)\right) \leq 4 \pi n \cdot q^{3+\sigma}_n \cdot \exp\left(2 \pi \cdot q_n \cdot \tilde{\rho}_n\right) \notag \\
& \leq 4\pi \cdot n \cdot q^{3 + \sigma}_n \cdot \exp\left(4 \pi \cdot q^3_n \cdot \exp\left(2\pi \cdot q^2_n \cdot \left(1+n \cdot q^{\sigma}_n\right)\right)\right).
\end{align}

In order to be able to apply the criterion for weak mixing \ref{prop:wm} we have the requirement \ref{eq:C1}: $q_{n+1} \geq \left\|DH_n\right\|^2_0$. Using the above calculations we obtain
\begin{equation*}
\left\|DH_n\right\|_0 \leq  2! \cdot \left\|DH_{n-1}\right\|_0 \cdot \left\|Dh_n\right\|_0 \leq 2 \cdot q^{0.5}_n \cdot 4\pi \cdot n \cdot q^{3+\sigma}_n   \leq 8 \pi \cdot n \cdot q^{3.5 + \sigma}_n.
\end{equation*}
So we demand
\begin{equation} \tag{B3} \label{eq:d0c}
q_{n+1}\geq  64 \pi^2 \cdot n^2 \cdot q^{8}_n.
\end{equation}

\subsection{Proof of Convergence}
As a preparatory result we prove the subsequent Lemma:
\begin{lem} \label{lem:prep}
Let $n\in \mathbb{N}$, $m \in \mathbb{N}_0$. Under the condition
\begin{equation*}
\left\|h_n \circ R^m_{\alpha_{n+1}} \circ h^{-1}_n - R^m_{\alpha_n} \right\|_{\rho_{n-1}} < \frac{1}{2^n \cdot \left\|Dh_1\right\|_{\rho_1+1} \cdot ... \cdot \left\|Dh_{n-1}\right\|_{\rho_{n-1}+1}}
\end{equation*}
we have $d_{\rho}\left(f^m_n, f^m_{n-1}\right) < \frac{1}{2^n}$.
\end{lem}
\begin{proof}
We introduce the functions $\psi_{n,k} \coloneqq h_{k+1} \circ ... \circ h_n \circ R^{m}_{\alpha_{n+1}} \circ h^{-1}_n ... \circ h^{-1}_{k+1}$ in case of $1 \leq k \leq n-1$ and $\psi_{n,n} = R^m_{\alpha_{n+1}}$. With these we have $f^m_n= h_1 \circ ... \circ h_k \circ \psi_{n,k} \circ h^{-1}_k \circ ... \circ h^{-1}_1$. By induction on $k \in \mathbb{N}$ in the range $1 \leq k \leq n-1$ we prove: \\
\textbf{Fact: } Under the condition $\left\|\psi_{n,k} \circ H^{-1}_k - \psi_{n-1,k} \circ H^{-1}_k\right\|_{\rho} < \frac{1}{2^n \cdot \left\|Dh_1\right\|_{\rho_1+1} \cdot ... \cdot \left\|Dh_{k}\right\|_{\rho_{k}+1}}$ we have \begin{equation*}
\left\|h_1 \circ ... \circ h_k \circ \psi_{n,k} \circ h^{-1}_k \circ ... \circ h^{-1}_1 - h_1 \circ ... \circ h_k \circ \psi_{n-1,k} \circ h^{-1}_k \circ ... \circ h^{-1}_1\right\|_{\rho} < \frac{1}{2^n}.
\end{equation*} 
\begin{itemize}
	\item \textit{Start: $k=1$} \\
	At first we note $h^{-1}_1\left(A^{\rho}\right)\subseteq A^{\rho_1}$. By our assumption we have 
	\begin{equation*}
	\left\|\psi_{n,1} \circ h^{-1}_1 - \psi_{n-1,1}\circ h^{-1}_1\right\|_{\rho} < \frac{1}{2^n \cdot \left\|Dh_1\right\|_{\rho_1+1}} < 1.
	\end{equation*}
	So a sufficient condition for our claim is given by 
	\begin{equation*}
	\left\|Dh_1\right\|_{\rho_1 +1} \cdot \left\|\psi_{n,1} \circ h^{-1}_1 - \psi_{n-1,1} \circ h^{-1}_1\right\|_{\rho}< \frac{1}{2^n},
	\end{equation*}
	which is satisfied by our requirements.
	\item \textit{Induction hypothesis: } The claim holds for $1\leq k-1 \leq n-2$.
	\item \textit{Induction step: $k-1 \rightarrow k$} \\
	Using the induction hypothesis the proximity 
	\begin{align*}
	\left\|\psi_{n,k-1} \circ H^{-1}_{k-1}-\psi_{n-1,k-1} \circ H^{-1}_{k-1}\right\|_{\rho} & = \left\|h_k \circ \psi_{n,k} \circ H^{-1}_k-h_k \circ \psi_{n-1,k} \circ H^{-1}_k \right\|_{\rho} \\
	& <\frac{1}{2^n \cdot \left\|Dh_1\right\|_{\rho_1+1} \cdot ... \cdot \left\|Dh_{k-1}\right\|_{\rho_{k-1}+1}} 
	\end{align*}
	is sufficient to prove the claim. Since $\left\|\psi_{n,k} \circ H^{-1}_k - \psi_{n-1,k} \circ H^{-1}_{k}\right\|_{\rho} < 1$ this is fulfilled if 
	\begin{equation*}
	\left\|Dh_k\right\|_{\rho_{k}+1} \cdot \left\|\psi_{n,k} \circ H^{-1}_k- \psi_{n-1,k}\circ H^{-1}_k\right\|_{\rho} <\frac{1}{2^n \cdot \left\|Dh_1\right\|_{\rho_1+1} \cdot ... \cdot \left\|Dh_{k-1}\right\|_{\rho_{k-1}+1}}.
	\end{equation*}
	By our assumption on $\left\|\psi_{n,k}\circ H^{-1}_k - \psi_{n-1,k}\circ H^{-1}_k\right\|_{\rho}$ the claim is true.
\end{itemize}
In the opposite direction we show that our assumption on $\left\|h_n \circ R^m_{\alpha_{n+1}} \circ h^{-1}_n - R^m_{\alpha_n} \right\|_{\rho_{n-1}}$ implies the conditions $\left\|\psi_{n,k}\circ H^{-1}_k - \psi_{n-1,k}\circ H^{-1}_k\right\|_{\rho} < \frac{1}{2^n \cdot \left\|Dh_1\right\|_{\rho_1+1} \cdot ... \cdot \left\|Dh_{k}\right\|_{\rho_{k}+1}}$:
\begin{itemize}
	\item \textit{Start: $k=n-1$} \\
	The condition on $\left\|\psi_{n,n-1}\circ H^{-1}_{n-1} - \psi_{n-1,n-1}\circ H^{-1}_{n-1}\right\|_{\rho} \leq \left\|h_n \circ R^m_{\alpha_{n+1}} \circ h^{-1}_n - R^m_{\alpha_n} \right\|_{\rho_{n-1}}$ is exactly the supposition of the Lemma.
	\item \textit{Induction hypothesis: } The claim holds for $2\leq k \leq n-1$.
	\item \textit{Induction step: $k \rightarrow k-1$} \\
	We estimate with the aid of our induction hypothesis
	\begin{align*}
	\left\|\psi_{n,k-1} \circ H^{-1}_{k-1} - \psi_{n-1,k-1} \circ H^{-1}_{k-1} \right\|_{\rho} &= \left\|h_k \circ \psi_{n,k} \circ H^{-1}_k - h_k \circ \psi_{n-1,k} \circ H^{-1}_k\right\|_{\rho} \\
	& \leq \left\|Dh_k\right\|_{\rho_k+1} \cdot \left\|\psi_{n,k} \circ H^{-1}_k - \psi_{n-1,k} \circ H^{-1}_k\right\|_{\rho} \\
	& < \left\|Dh_k\right\|_{\rho_k+1} \cdot \frac{1}{2^n \cdot \left\|Dh_1\right\|_{\rho_1+1} \cdot ... \cdot \left\|Dh_{k}\right\|_{\rho_{k}+1}} \\
	& = \frac{1}{2^n \cdot \left\|Dh_1\right\|_{\rho_1+1} \cdot ... \cdot \left\|Dh_{k-1}\right\|_{\rho_{k-1}+1}}.
	\end{align*}
\end{itemize}
Hence, the requiremts of the fact are met and the Lemma is proven.
\end{proof}

Now we are able to deduce the aimed statement on convergence of $\left(f_n\right)_{n \in \mathbb{N}}$ in the Diff$^{\omega}_{\rho}\left(\mathbb{T}^2\right)$-topology:
\begin{lem} \label{lem:convc}
We assume that
\begin{equation} \tag{B1'} \label{eq:alphaB}
\left|\alpha_{n+1} - \alpha_n \right| < \frac{1}{2^n\cdot \left\|Dh_1\right\|_{\rho_1+1} \cdot ... \cdot \left\|Dh_{n-1}\right\|_{\rho_{n-1}+1} \cdot 4 \pi n \cdot q^{4 + \sigma}_n \cdot \exp\left(4 \pi n \cdot q^{1+\sigma}_n \cdot \tilde{\rho}_{n-1}\right)}.
\end{equation}
Then the diffeomorphisms $f_n = H_n \circ R_{\alpha_{n+1}} \circ H^{-1}_n$ satisfy:
\begin{itemize}
	\item The sequence $\left(f_n\right)_{n \in \mathbb{N}}$ converges in the Diff$^{\omega}_{\rho}\left(\mathbb{T}^2\right)$-topology to a measure-preserving diffeomorphism $f$.
	\item We have for every $n \in \mathbb{N}$ and $m \leq q_{n+1}$:
	\begin{equation*}
	d_0\left(f^{m}, f^{m}_n\right) < \frac{1}{2^n}.
	\end{equation*}
\end{itemize}
\end{lem}

\begin{proof}
At first we introduce for $m \in \mathbb{N}$ the function 
\begin{equation*}
T_m(z) = \cos\left(2 \pi \cdot q_n \cdot \left(z + m \cdot \alpha_{n+1}\right)\right)-\cos\left(2 \pi \cdot q_n \cdot z\right)
\end{equation*}
and exploiting the relation $\cos\left(x\right)-\cos\left(y\right)= 2 \cdot \sin\left(\frac{x+y}{2}\right) \cdot \sin\left(\frac{y-x}{2}\right)$ we can estimate for every $s\geq 0$:
\begin{align*}
\left\|T_m\right\|_s & = \left\|2 \cdot \sin\left(\pi \cdot q_n \cdot \left(2z + m \cdot \alpha_{n+1}\right)\right) \cdot \sin\left(\pi \cdot q_n \cdot m \cdot \alpha_{n+1}\right)\right\|_s \\
& = 2 \cdot \left\|\frac{1}{2\imath}\left(e^{\imath \pi \cdot q_n \cdot \left(2z +m \cdot \alpha_{n+1}\right)}-e^{-\imath \pi \cdot q_n \cdot \left(2z +m \cdot \alpha_{n+1}\right)}\right)\right\|_s \cdot \left|\sin\left(\pi \cdot q_n \cdot m \cdot\alpha_{n+1}\right)\right| \\
& \leq 2 \cdot \left\|e^{2 \pi \imath \cdot q_n \cdot z}\right\|_s \cdot \left|\sin\left(\pi \cdot q_n \cdot m \cdot\left(\alpha_{n+1}-\alpha_n\right)\right)\right| \\
& \leq 2 \cdot \left\|e^{2 \pi \imath \cdot q_n \cdot z}\right\|_s \cdot \pi \cdot q_n \cdot m \cdot \left|\alpha_{n+1}-\alpha_n\right|,
\end{align*}
where we made use of $\left|\sin\left(x\right)\right| \leq \left|x\right|$ in the last step. \\
Using this map $T_m$ we compute
\begin{align*}
& h_n \circ R^m_{\alpha_{n+1}} \circ h^{-1}_n\left(\theta,r\right) = h_n \circ R^m_{\alpha_{n+1}}\left( \theta - \left[nq^{\sigma}_n\right] \cdot r, r - q^2_n \cdot \cos\left(2 \pi \cdot q_n \cdot \left(\theta - \left[nq^{\sigma}_n\right] \cdot r\right)\right)\right) \\
&= h_n\left(\theta + m \cdot\alpha_{n+1} - \left[nq^{\sigma}_n\right] \cdot r, r - q^2_n \cdot \cos\left(2 \pi \cdot q_n \cdot \left(\theta - \left[nq^{\sigma}_n\right] \cdot r\right)\right)\right) \\
&= g_n\left(\theta + m \cdot \alpha_{n+1} - \left[nq^{\sigma}_n\right] \cdot r, r + q^2_n \cdot T_m\left(\theta-\left[nq^{\sigma}_n\right] \cdot r\right)\right) \\
&=\left(\theta + m \cdot \alpha_{n+1} + \left[nq^{\sigma}_n\right] \cdot q^2_n \cdot T_m\left(\theta-\left[nq^{\sigma}_n\right] \cdot r\right), r + q^2_n \cdot T_m\left(\theta-\left[nq^{\sigma}_n\right] \cdot r\right)\right).
\end{align*}
Then we have
\begin{align*}
& h_n \circ R^m_{\alpha_{n+1}} \circ h^{-1}_n\left(\theta,r\right) - R^m_{\alpha_n}\left(\theta,r\right) \\
= & \left(m \cdot \alpha_{n+1} - m \cdot \alpha_n + \left[nq^{\sigma}_n\right] \cdot q^2_n \cdot T_m\left(\theta-\left[nq^{\sigma}_n\right] \cdot r\right),q^2_n \cdot T_m\left(\theta-\left[nq^{\sigma}_n\right] \cdot r\right)\right).
\end{align*}
Thus, we can estimate for $m \leq q_n$:
\begin{align*}
& \left\|h_n \circ R^m_{\alpha_{n+1}} \circ h^{-1}_n \circ H^{-1}_{n-1} - R^m_{\alpha_n}\circ H^{-1}_{n-1} \right\|_{\rho} \leq \left\|h_n \circ R^m_{\alpha_{n+1}} \circ h^{-1}_n  - R^m_{\alpha_n} \right\|_{\rho_{n-1}} \\
& \leq 2 \cdot \left[nq^{\sigma}_n\right] \cdot q^2_n \cdot \left\|T_m\left(\theta-\left[nq^{\sigma}_n\right] \cdot r\right)\right\|_{\rho_{n-1}} \\
& \leq 2 \cdot \left[nq^{\sigma}_n\right] \cdot q^2_n \cdot 2 \cdot \left\|e^{2 \pi \imath \cdot q_n \cdot \left(\theta-\left[nq^{\sigma}_n\right] \cdot r\right)}\right\|_{\rho_{n-1}} \cdot \pi \cdot q_n \cdot m \cdot \left|\alpha_{n+1}-\alpha_n\right| \\
& \leq 4 \cdot \pi \cdot n \cdot q^{3 + \sigma}_n \cdot e^{4 \pi \cdot n \cdot q^{1+\sigma}_n \cdot \rho_{n-1}} \cdot q_n \cdot \left|\alpha_{n+1}-\alpha_n\right| \\
& < \frac{1}{2^n \cdot \left\|Dh_1\right\|_{\rho_1+1} \cdot ... \cdot \left\|Dh_{n-1}\right\|_{\rho_{n-1}+1}}.
\end{align*}
So the prerequisites of Lemma \ref{lem:prep} are fulfilled and we conclude $d_{\rho}\left(f^m_n, f^m_{n-1}\right)< \frac{1}{2^n}$. In the same spirit as in the proof of Lemma \ref{lem:conv} we can show the convergence of $\left(f_n\right)_{n \in \mathbb{N}}$ and the second property.
\end{proof}

Now we formulate the next requirement on the sequence $\left(q_n\right)_{n \in \mathbb{N}}$:
\begin{equation} \tag{B4'} \label{eq:DHc3}
q_{n+1} \geq \left\|Dh_1\right\|_{\rho_1+1} \cdot ... \cdot \left\|Dh_{n-1}\right\|_{\rho_{n-1}+1} \cdot \left\|Dh_{n}\right\|_{\rho_{n}+1}.
\end{equation}
By the requirement $\left\|Dh_1\right\|_{\rho_1+1} \cdot ... \cdot \left\|Dh_{n-1}\right\|_{\rho_{n-1}+1} \leq q_n$ as well as equation \ref{align:Dh} condition \ref{eq:DHc3} is satisfied if we demand
\begin{equation} \tag{B4} \label{eq:DHc2}
q_{n+1} \geq 4 \pi \cdot n \cdot q^{4 + \sigma}_n \cdot \exp\left(4 \pi \cdot q^3_n \cdot \exp\left(2 \pi \cdot q^2_n \cdot \left(1+n \cdot q^{\sigma}_n\right)\right)\right).
\end{equation}

Since $\left|\alpha_{n+1} - \alpha_n \right|=\frac{a_n}{q_n \cdot \tilde{q}_{n+1}} \leq \frac{1}{\tilde{q}_{n+1}}$ condition \ref{eq:alphaB} yields the requirement
\begin{equation*}
\tilde{q}_{n+1} \geq 2^n \cdot \left\|Dh_1\right\|_{\rho_1+1} \cdot ... \cdot \left\|Dh_{n-1}\right\|_{\rho_{n-1}+1} \cdot 4 \pi \cdot n \cdot q^{4+\sigma}_n \cdot \exp\left(4\pi \cdot n \cdot q^{1+\sigma}_n \cdot \tilde{\rho}_{n-1}\right).
\end{equation*}
Using $\tilde{\rho}_{n-1} \leq q_n$ (see condition \ref{eq:rho}) and $\left\|Dh_1\right\|_{\rho_1+1} \cdot ... \cdot \left\|Dh_{n-1}\right\|_{\rho_{n-1}+1} \leq q_n$ (see condition \ref{eq:DHc3}) this requirement is satisfied if we demand
\begin{equation} \tag{B1} \label{eq:convc}
\tilde{q}_{n+1} \geq 2^n \cdot 4 \pi \cdot n \cdot q^{5+\sigma}_n \cdot \exp\left(4\pi \cdot n \cdot q^{2+\sigma}_n \right).
\end{equation}

By collecting all the prerequisites \ref{eq:convc}, \ref{eq:rhoc}, \ref{eq:d0c}, \ref{eq:DHc2} on the sequence $\left(q_n\right)_{n \in \mathbb{N}}$ we demand:
\begin{equation*}
q_{n+1} \geq 2^n \cdot 64\pi^2 \cdot n^2 \cdot q^{8}_n \cdot \exp\left(4\pi \cdot n \cdot q^3_n \cdot \exp\left(2 \pi \cdot q^2_n \cdot \left(1+n \cdot q^{\sigma}_n\right)\right)\right).
\end{equation*}
Since $\tilde{q}_{n+1} = \frac{q_{n+1}}{q_n}$, $q_n = q_{n-1} \cdot \tilde{q}_n \leq \tilde{q}^2_n$ and $0< \sigma < \frac{1}{2}$ we obtain the following sufficient condition on the growth rate of the rigidity sequence $\left(\tilde{q}_n\right)_{n \in \mathbb{N}}$:
\begin{equation*}
\tilde{q}_{n+1} \geq 2^n \cdot 64\pi^2 \cdot n^2 \cdot \tilde{q}^{14}_n \cdot \exp\left(4\pi \cdot n \cdot \tilde{q}^6_n \cdot \exp\left(2 \pi \cdot \tilde{q}^4_n \cdot \left(1+n \cdot \tilde{q}_n\right)\right)\right).
\end{equation*}

\section{Proof of weak mixing} \label{section:wm}
By the same approach as in \cite{FS} we want to apply Proposition \ref{prop:wm}. For this purpose, we introduce a sequence $\left(m_n\right)_{n \in \mathbb{N}}$ of natural numbers $m_n \leq q_{n+1}$ in subsection \ref{subsection:m_n} and a sequence $\left(\eta_n\right)_{n \in \mathbb{N}}$ of standard partial decompositions in subsection \ref{subsection:eta}. Finally, we show that the diffeomorphism $\Phi_n \coloneqq \phi_n \circ R^{m_n}_{\alpha_{n+1}} \circ \phi^{-1}_n$ $\left(\frac{1}{nq_n}, 0, \frac{1}{n}\right)$-distributes the elements of this partition.

\subsection{Choice of the mixing sequence $\left(m_n\right)_{n \in \mathbb{N}}$} \label{subsection:m_n}
By condition \ref{eq:C1diff} resp. \ref{eq:d0c} our chosen sequence $\left(q_{n}\right)_{n \in \mathbb{N}}$ satisfies
\begin{equation} \tag{C2} \label{eq:m}
q_{n+1} \geq q^8_n.
\end{equation}
Define
\begin{align*}
m_n & = \min \left\{ m \leq q_{n+1} \ \  : \ \ m \in \mathbb{N},\ \  \inf_{k \in \mathbb{Z}} \left| m \cdot \frac{p_{n+1}}{q_{n+1}} - \frac{1}{2 \cdot q_n} + \frac{k}{q_n}\right| \leq \frac{q_n}{q_{n+1}}\right\} \\
& = \min \left\{ m \leq q_{n+1} \ \ : \ \ m \in \mathbb{N},\ \ \inf_{k \in \mathbb{Z}} \left| m \cdot \frac{q_n \cdot p_{n+1}}{q_{n+1}} - \frac{1}{2} + k \right| \leq \frac{q^2_n}{q_{n+1}}\right\}
\end{align*}
\begin{lem}
The set $M_n := \left\{ m \leq q_{n+1} \ \ : \ \ m \in \mathbb{N},\ \ \inf_{k \in \mathbb{Z}} \left| m \cdot \frac{q_n \cdot p_{n+1}}{q_{n+1}} - \frac{1}{2} + k \right| \leq \frac{q^2_n}{q_{n+1}}\right\}$ is nonempty for every $n \in \mathbb{N}$, i.e. $m_n$ exists.
\end{lem}
\begin{pr}
The number $\alpha_{n+1}$ was constructed by the rule $\frac{p_{n+1}}{q_{n+1}}= \frac{p_n}{q_n}-\frac{a_n}{q_n \cdot \tilde{q}_{n+1}}$, where $a_n \in \mathbb{Z}$, $1 \leq a_n \leq q_n$, i.e. $p_{n+1}= p_n \cdot \tilde{q}_{n+1} - a_n$ and $q_{n+1} = q_n \cdot \tilde{q}_{n+1}$. So $\frac{q_n \cdot p_{n+1}}{q_{n+1}}= \frac{p_{n+1}}{\tilde{q}_{n+1}}$ and the set $\left\{ j \cdot \frac{q_n \cdot p_{n+1}}{q_{n+1}} \ \  : \ \ j=1,2,...,q_{n+1} \right\}$ contains $\frac{\tilde{q}_{n+1}}{\text{gcd}\left(p_{n+1}, \tilde{q}_{n+1}\right)}$ different equally distributed points on $\mathbb{S}^1$. Hence, there are at least $\frac{\tilde{q}_{n+1}}{q_n}= \frac{q_{n+1}}{q^2_n}$ different such points and so for every $x \in \mathbb{S}^1$ there is a $j \in \left\{1,...,q_{n+1} \right\}$ such that 
\begin{equation*}
\inf_{k \in \mathbb{Z}} \left| x - j \cdot \frac{q_n \cdot p_{n+1}}{q_{n+1}} + k \right| \leq \frac{q^2_n}{q_{n+1}}.
\end{equation*}
In particular, this is true for $x=\frac{1}{2}$.
\end{pr}

\begin{rem} \label{rem:aneigid}
We define
\begin{equation*}
\Delta_n = \left(m_n \cdot \frac{p_{n+1}}{q_{n+1}} - \frac{1}{2 \cdot q_n}\right) \text{ mod } \frac{1}{q_n}.
\end{equation*}
By the above construction of $m_n$ it holds: $\left|\Delta_n \right| \leq \frac{q_n}{q_{n+1}}$. By \ref{eq:m} we get:
\begin{equation*}
\left|\Delta_n \right| \leq \frac{1}{q^7_n}.
\end{equation*}
\end{rem}

\subsection{Standard partial decomposition $\eta_n$} \label{subsection:eta}
In the following we will consider the set
\begin{equation*}
B_n = \bigcup^{2q_n-1}_{k=0}\left[\frac{k}{2q_n}-\frac{1}{2q^{1.5}_n}, \frac{k}{2q_n}+\frac{1}{2q^{1.5}_n}\right].
\end{equation*}
Our horizontal intervals belonging to the partial decomposition $\eta_n$ will lie outside $B_n$. To show that $\Phi_n$ $\left(\frac{1}{nq_n}, 0, \frac{1}{n}\right)$-distributes the elements of this partition we will need the subsequent result similar to the concept of ``uniformly stretching'' from \cite{F}:
\begin{lem} \label{lem:us}
Let $I=\left[a,b\right] \subset \mathbb{R}$ be an interval and $\psi: I \rightarrow \mathbb{R}$ be a strictly monotonic $C^2$-function. Furthermore, we denote $J \coloneqq \left[ \inf_{x \in I} \psi\left(x\right), \sup_{x \in I} \psi\left(x\right)\right]$. If $\psi$ satisfies
\begin{equation*}
\sup_{x \in I} \left|\psi''\left(x\right)\right| \cdot \lambda\left(I\right) \leq \varepsilon \cdot \inf_{x \in I} \left|\psi'\left(x\right)\right|,
\end{equation*}
then for any interval $\tilde{J} \subseteq J$ we have
\begin{equation*}
\left|\frac{\lambda\left(I \cap \psi^{-1}\left(\tilde{J}\right)\right)}{\lambda\left(I\right)}-\frac{\lambda\left(\tilde{J}\right)}{\lambda\left(J\right)}\right|\leq \varepsilon \cdot \frac{\lambda\left(\tilde{J}\right)}{\lambda\left(J\right)}.
\end{equation*}.
\end{lem}
\begin{proof}
We consider the case that $\psi$ is strictly increasing (the proof in the decreasing case is analogous), which implies $\psi'>0$ (due to our assumption $\sup_{x \in I} \left|\psi''\left(x\right)\right| \cdot \lambda\left(I\right) \leq \varepsilon \cdot \inf_{x \in I} \left|\psi'\left(x\right)\right|$) and $J = \left[\psi\left(a\right), \psi\left(b\right)\right]$. \\
Let $\tilde{J} = \left[ \psi\left(c\right),\psi\left(d\right)\right]$, where $a\leq c\leq d \leq b$. By the mean value theorem there are $\xi_1 \in \left[a,b\right]$ and $\xi_2 \in \left[c,d\right]$, such that $\psi\left(b\right)-\psi\left(a\right) = \psi'\left(\xi_1\right) \cdot \left(b-a\right)$ resp. $\psi\left(d\right)-\psi\left(c\right) = \psi'\left(\xi_2\right) \cdot \left(d-c\right)$. Applying the mean value theorem on $\psi'$ gives $\xi_3 \in \left[a,b\right]$ with $\left|\psi'\left(\xi_2\right)-\psi'\left(\xi_1\right)\right|=\left|\psi''\left(\xi_3\right)\right|\cdot \left|\xi_2 - \xi_1\right|$. Then we have:
\begin{equation*}
\left|\psi'\left(\xi_1\right)-\psi'\left(\xi_2\right)\right| \leq \sup_{x \in \left[a,b\right]} \left|\psi''\left(x\right)\right| \cdot \left|b-a\right| = \sup_{x \in \left[a,b\right]} \left|\psi''\left(x\right)\right| \cdot \lambda\left(I\right) \leq \varepsilon \cdot \inf_{x \in \left[a,b\right]} \left|\psi'\left(x\right)\right| \leq \varepsilon \cdot \left|\psi'\left(\xi_2\right)\right|.
\end{equation*}
Hereby, we obtain:
\begin{equation*}
\left|\frac{\psi'\left(\xi_1\right)}{\psi'\left(\xi_2\right)} - 1\right| \leq \varepsilon.
\end{equation*}
Since $\psi'>0$ this implies $1-\varepsilon \leq \frac{\psi'\left(\xi_1\right)}{\psi'\left(\xi_2\right)} \leq 1+\varepsilon$ and thus:
\begin{equation*}
\frac{\lambda\left(I\cap \psi^{-1}\left(\tilde{J}\right)\right)}{\lambda\left(I\right)} = \frac{d-c}{b-a}=\frac{\psi'\left(\xi_1\right) \cdot \left(\psi\left(d\right)-\psi\left(c\right)\right)}{\psi'\left(\xi_2\right) \cdot \left(\psi\left(b\right)-\psi\left(a\right)\right)} \leq \left(1+\varepsilon\right) \cdot \frac{\lambda\left(\tilde{J}\right)}{\lambda\left(J\right)}.
\end{equation*}
This implies
\begin{equation*}
\frac{\lambda\left(I \cap \psi^{-1}\left(\tilde{J}\right)\right)}{\lambda\left(I\right)}-\frac{\lambda\left(\tilde{J}\right)}{\lambda\left(J\right)} \leq \varepsilon \cdot \frac{\lambda\left(\tilde{J}\right)}{\lambda\left(J\right)}.
\end{equation*}
Similarly we obtain the estimate from below and the claim follows.
\end{proof} 

By the explicit definitions of the conjugation maps the transformation $\Phi_n = \phi_n \circ R^{m_n}_{\alpha_{n+1}} \circ \phi^{-1}_n$ has a lift of the form
\begin{equation*}
\Phi_n\left(\theta,r\right)= \left( \theta + m_n \cdot \alpha_{n+1}, r+ \psi_n\left(\theta\right)\right)
\end{equation*}
with
\begin{equation*}
\psi_n\left(\theta\right)= q^2_n \cdot \left( \cos\left(2 \pi \left(q_n \theta + m_n q_n \alpha_{n+1}\right)\right)- \cos\left(2 \pi q_n \theta\right)\right).
\end{equation*}
We examine this map $\psi_n$:
\begin{lem} \label{lem:psi}
The map $\psi_n$ satisfies:
\begin{equation*}
\inf_{ \theta \in \mathbb{T} \setminus B_n} \left| \psi_n'\left(\theta\right)\right| \geq q^{2.5}_n \ \ \ \ \ \text{and} \ \ \ \ \ \ \sup_{ \theta \in \mathbb{T} \setminus B_n} \left| \psi_n''\left(\theta\right)\right| \leq 9 \pi^2 q^4_n.
\end{equation*}
\end{lem}

\begin{proof}
Using the relation $\cos\left(2\pi q_n \theta +\pi\right) = - \cos\left(2 \pi q_n \theta\right)$ and the map  
\begin{equation*}
\sigma_n\left(\theta\right)=q^2_n \cdot \left(\cos\left(2 \pi q_n\left(\theta+m_n \alpha_{n+1}\right)\right)-\cos\left(2\pi q_n\left( \theta+\frac{1}{2q_n}\right)\right)\right)
\end{equation*}
we can write $\psi_n\left(\theta\right)= -2q^2_n \cdot \cos\left(2\pi q_n \theta\right)+\sigma_n\left(\theta\right)$. For this map $\sigma_n$ we compute:
\begin{equation*}
\sigma_n'\left(\theta\right)= 2\pi \cdot q^3_n \cdot \left(-\sin\left(2 \pi q_n \left(\theta+m_n \cdot \alpha_{n+1}\right)\right)+\sin\left(2 \pi q_n \left(\theta + \frac{1}{2q_n}\right)\right)\right). 
\end{equation*}
Applying the mean value theorem on the function $\varphi_a\left(\xi\right) \coloneqq -2 \pi \cdot q^3_n \cdot \sin\left(2 \pi q_n \xi\right)$ we obtain:
\begin{equation*}
\left|\sigma_n'\left(\theta\right)\right| \leq 2 \pi q^3_n \cdot \max_{\xi \in \mathbb{T}} \left|-2 \pi q_n \cdot \cos\left(2 \pi q_n \xi\right)\right| \cdot  \Delta_n \leq \left(2 \pi\right)^2 \cdot q^4_n \cdot \frac{q_n}{q_{n+1}} \leq \left(2 \pi\right)^2 \cdot q^4_n \cdot \frac{1}{q^7_{n}} < 1.
\end{equation*}
On the other hand, on the set $\mathbb{T} \setminus B_n$ it holds:
\begin{align*}
\inf_{\theta \in \mathbb{T} \setminus B_n} \left|\sin \left(2 \pi q_n \theta\right)\right| & = \inf_{\theta=\tilde{\theta}+\frac{k}{2q_n}, k \in \mathbb{Z}, \tilde{\theta} \in \left[\frac{1}{2q^{1.5}_n}, \frac{1}{2q_n}-\frac{1}{2q^{1.5}_n}\right]} \left|\sin \left(2 \pi q_n \theta\right)\right| \\
& = \inf_{\tilde{\theta} \in \left[\frac{1}{2q^{1.5}_n}, \frac{1}{2q_n}-\frac{1}{2q^{1.5}_n}\right]} \left|\sin \left(2 \pi q_n \tilde{\theta}\right)\right| \\
& = \inf_{\tilde{\theta} \in \left[\frac{1}{2q^{1.5}_n}, \frac{1}{4q_n}\right]} \left|\sin \left(2 \pi q_n \tilde{\theta}\right)\right| \geq \frac{1}{2} \cdot 2 \pi q_n \cdot \frac{1}{2q^{1.5}_n} = \frac{\pi}{2} \cdot q^{-0.5}_n > q^{-0.5}_n
\end{align*}
with the aid of the estimate $\sin(x)\geq \frac{1}{2}x$ for $x \in \left[0,\frac{\pi}{2}\right]$. Thus, we have:
\begin{equation*}
\inf_{ \theta \in \mathbb{T} \setminus B_n} \left| \psi_n'\left(\theta\right)\right| \geq 4 \pi q^3_n \cdot \inf_{\theta \in \mathbb{T} \setminus B_n} \left|\sin \left(2 \pi q_n \theta\right)\right| - \sup_{\theta \in \mathbb{T} \setminus B_n} \left|\sigma_n'\left(\theta\right)\right| \geq 4 \pi q^3_n \cdot q^{-0.5}_n -1 \geq q^{2.5}_n.
\end{equation*}
In order to estimate $\psi_n''$ we compute
\begin{equation*}
\sigma_n''\left(\theta\right)= \left(2\pi\right)^2 \cdot q^4_n \cdot \left(-\cos\left(2 \pi q_n\left( \theta+m_n \cdot \alpha_{n+1}\right)\right)+\cos\left(2 \pi q_n \left(\theta + \frac{1}{2q_n}\right)\right)\right)
\end{equation*}
and use the mean value theorem on $\varphi_b\left(\xi\right) \coloneqq -\left(2 \pi\right)^2 \cdot q^4_n \cdot \cos\left(2 \pi q_n \xi\right)$:
\begin{equation*}
\left|\sigma_n''\left(\theta\right)\right| \leq \left(2 \pi\right)^2 \cdot q^4_n \cdot \max_{\xi \in \mathbb{T}} \left| 2 \pi q_n \cdot \sin\left(2 \pi q_n \xi\right)\right|  \cdot \Delta_n \leq \left(2 \pi\right)^3 \cdot q^5_n \cdot \frac{q_n}{q_{n+1}} < 1.
\end{equation*}
Then we obtain:
\begin{equation*}
\sup_{\theta \in \mathbb{T} \setminus B_n} \left|\psi_n''\left(\theta\right)\right| \leq \sup_{\theta \in \mathbb{T} \setminus B_n} \left|2 \cdot \left(2 \pi\right)^2 \cdot q^4_n \cdot \cos\left(2 \pi q_n \theta\right)\right|+\sup_{\theta \in \mathbb{T} \setminus B_n} \left|\sigma_n''\left(\theta\right)\right| \leq 8 \pi^2 \cdot q^4_n + 1 \leq 9 \pi^2 \cdot q^4_n.
\end{equation*}
\end{proof}

The aimed standard partial decomposition $\eta_n$ of $\mathbb{T}^2$ is defined in the following way: \\
Let $\hat{\eta}_n = \left\{ \hat{I}_{n,i}\right\}$ be the partial partition of $\mathbb{T}\setminus B_n$ consisting of all the disjoint intervals $\hat{I}_{n,i}$ such that there is $k \in \mathbb{Z}$: $\psi_n\left(\hat{I}_{n,i}\right) = k + \left[0,1\right)$. Then we define
\begin{equation*}
\eta_n = \left\{ \hat{I} \times \left\{r\right\} \;:\; \hat{I} \in \hat{\eta}_n,\;r \in \mathbb{T}\right\}.
\end{equation*}
Note that we have $\pi_r\left(\Phi_n\left(I_n\right)\right) = \mathbb{T}$ for every $I_n \in \eta_n$.
\begin{lem} \label{lem:eta}
For any partition element $\hat{I}_n \in \hat{\eta}_n$ we have $\lambda\left(\hat{I}_n\right) \leq q^{-2.5}_n$. Moreover, it holds: $\eta_n \rightarrow \varepsilon$.
\end{lem}
\begin{proof}
By Lemma \ref{lem:psi} we have $\inf_{ \theta \in \mathbb{T} \setminus B_n} \left| \psi_n'\left(\theta\right)\right| \geq q^{2.5}_n$. Therefore, $\lambda\left(\hat{I}_n\right)\leq q^{-2.5}_n$ for any $I_n \in \eta_n$. Hence, the length of the elements of $\eta_n$ goes to zero as $n \rightarrow \infty$. Thus, in order to prove $\eta_n \rightarrow \varepsilon$ we have only to check that the total measure of the partial decompositions $\eta_n$ goes to $1$ as $n \rightarrow \infty$. Since the elements of $\hat{\eta}_n$ are contained in $\mathbb{T} \setminus B_n$ and have to satisfy the additional requirement $\psi_n\left(\hat{I}_{n}\right) = k + \left[0,1\right)$ for $k \in \mathbb{Z}$, on both sides around the set $B_n$ there is an area without partition elements. So the total measure of $\eta_n$ can be estimated as follows:
\begin{align*}
\sum_{\hat{I}_n \in \hat{\eta}_n} \lambda\left(\hat{I}_n\right) & \geq 1 - \lambda\left(B_n\right) - 2 \cdot 2 q_n \cdot \max_{\hat{I}_n \in \hat{\eta}_n} \lambda\left(\hat{I}_n\right) \\
& \geq 1- 2q_n \cdot \left(q^{-1.5}_n + 2q^{-2.5}_n\right) > 1 - 3 \cdot q^{-0.5}_n
\end{align*}
and this approaches $1$ as $n \rightarrow \infty$.
\end{proof} 

\subsection{Application of the criterion for weak mixing}
In order to apply the criterion for weak mixing we check that the constructed diffeormorphism $f = \lim_{n\rightarrow \infty} f_n$ fulfil the requirements: 
\begin{itemize}
	\item By Lemma \ref{lem:conv}, 2., resp. \ref{lem:convc}, 2.: $d_0\left(f^{m_n}, f^{m_n}_n\right)< \frac{1}{2^n}$, because $m_n \leq q_{n+1}$ by definition.
	\item Because of the requirement \ref{eq:C1diff} resp. \ref{eq:d0c} on the number $q_n$ we have \ref{eq:C1}.
	\item By Lemma \ref{lem:eta} we have $\eta_n \rightarrow \varepsilon$ and the length of the horizontal interval is at most $q^{-2.5}_n$.
\end{itemize}
Finally, the next Lemma proves the last remaining property:
\begin{lem}
Let $I_n \in \eta_n$. Then $\Phi_n$ $\left(\frac{1}{nq_n},0, \frac{1}{n}\right)$-distributes $I_n$.
\end{lem} 

\begin{proof}
The partial partition $\eta_n$ was chosen in such a way that $\pi_r\left(\Phi_n\left(I_n\right)\right) = \mathbb{T}$. Hence, $\delta$ can be taken equal to $0$. \\
Using the form of the lift $\Phi_n\left(\theta.r\right)= \left( \theta + m_n \cdot \alpha_{n+1}, r+ \psi_n\left(\theta\right)\right)$ we observe that $\Phi_n\left(I_n\right)$ is contained in the vertical strip $\left(I_n + m_n \cdot \alpha_{n+1}\right) \times \mathbb{T}$ for every $I_n \in \eta_n$. Since the length of $I_n$ is at most $\frac{1}{q^{2.5}_n} <\frac{1}{nq_n}$ by Lemma \ref{lem:eta}, we can take $\gamma = \frac{1}{nq_n}$. \\
Recall that an element $I_n \in \eta_n$ has the form $\hat{I} \times \left\{r\right\}$ for some $r \in \mathbb{T}$ and an interval $\hat{I} \in \hat{\eta}_n$ contained in $\mathbb{T} \setminus B_n$ with $\lambda\left(\hat{I}\right)\leq q^{-2.5}_n$ (see Lemma \ref{lem:eta}). Then Lemma \ref{lem:psi} implies the estimate
\begin{equation*}
\frac{\sup_{\theta \in \hat{I}} \left|\psi_n ''\left(\theta\right)\right| }{\inf_{\theta \in \hat{I}} \left|\psi_n '\left(\theta\right)\right|} \cdot \lambda\left(\hat{I}\right) \leq \frac{9 \pi^2 q^4_n}{q^{2.5}_n} \cdot q^{-2.5}_n = \frac{9\pi^2}{q_n} < \frac{1}{n}.
\end{equation*}
Then we can apply Lemma \ref{lem:us} on $\psi_n$ and $\hat{I}$ with $\varepsilon=\frac{1}{n}$. Moreover, we note that for any $\tilde{J} \subseteq J = \mathbb{T}$ the fact $\Phi_n\left(\theta,r\right) \in \mathbb{T} \times \tilde{J}$ is equivalent to $\psi_n\left(\theta\right) \in \tilde{J}-r \coloneqq \left\{ j-r \;:\;j \in \tilde{J}\right\}$. Combining these both results we conclude:
\begin{equation*}
\left|\frac{\lambda \left( I_n \cap \Phi^{-1}_n\left( \mathbb{T} \times \tilde{J}\right)\right)}{\lambda\left(I_n\right)}-\frac{\lambda\left(\tilde{J}\right)}{\lambda\left(J\right)}\right| = \left| \frac{\lambda\left(\hat{I} \cap \psi^{-1}_n\left(\tilde{J}-r\right)\right)}{\lambda\left(\hat{I}\right)} - \frac{\lambda\left(\tilde{J}\right)}{\lambda\left(J\right)} \right| \leq \frac{1}{n} \cdot \frac{\lambda\left(\tilde{J}\right)}{\lambda\left(J\right)}.
\end{equation*}
Thus, we can choose $\varepsilon = \frac{1}{n}$ in the definition of $\left(\gamma, \delta, \varepsilon\right)$-distribution.
\end{proof}

Then we can apply Proposition \ref{prop:wm} and conclude that the constructed diffeomorphisms are weak mixing.

\section{Proof of the Corollaries}
By some estimates we show that the assumptions of the Corollaries are enough to fulfil the requirements of the corresponding theorem.
 
\subsection{Proof of Corollary \ref{cor:Kor}} \label{subsection:prKor}
We recall the assumptions $\tilde{q}_1 \geq 108 \pi$ and $\tilde{q}_{n+1} \geq \tilde{q}^{\tilde{q}_n}_n$ on the sequence $\left(\tilde{q}_n\right)_{n \in \mathbb{N}}$. \\
\textbf{Claim: } Under these assumptions the numbers $\tilde{q}_n$ satisfy $\tilde{q}_n \geq 4 \pi \cdot n \cdot \left(n+2\right)^{n+2}$. \\
\textbf{Proof with the aid of complete induction: }
\begin{itemize}
	\item \textit{Start $n=1$: } $\tilde{q}_1 \geq 108 \pi = 4 \pi \cdot \left(1+2\right)^{1+2}$
	\item \textit{Assumption: } The claim is true for $n \in \mathbb{N}$.
	\item \textit{Induction step $n \rightarrow n+1$: } We calculate
	\begin{align*}
	\tilde{q}_{n+1} & \geq \tilde{q}^{\tilde{q}_n}_n \geq \left(4 \pi \cdot n \cdot \left(n+2\right)^{n+2}\right)^{4 \pi \cdot n \cdot \left(n+2\right)^{n+2}} \\
	& \geq 4 \pi \cdot \left(n \cdot \left(n+2\right)\right)^{\pi \cdot n \cdot \left(n+2\right)^{n+2}} \cdot \left(\left(n+2\right)^{n+1}\right)^{3 \pi \cdot n \cdot \left(n+2\right)^{n+2}} \\
	& \geq 4 \pi \cdot \left(n+1\right) \cdot \left(n+3\right)^{n+3},
	\end{align*}
	where we used the relation $\left(n+2\right)^{n+1} \geq n+3$ in the last step.
\end{itemize}
Hereby, we have
\begin{align*}
\tilde{q}_{n+1} & \geq \tilde{q}^{\tilde{q}_n}_n \geq \tilde{q}^{4 \pi \cdot n \cdot \left(n+2\right)^{n+2}}_n = \tilde{q}^{\pi \cdot n \cdot \left(n+2\right)^{n+2}}_n \cdot \tilde{q}^{3 \pi \cdot n \cdot \left(n+2\right)^{n+2}}_n \\
& \geq \left(4 \pi \cdot n \cdot \left(n+2\right)^{n+2}\right)^{\pi \cdot n \cdot \left(n+2\right)^{n+2}} \cdot \tilde{q}^{2 \cdot \left(n+2\right)^{n+2}}_n \\
& \geq \left(4 \pi \cdot n \cdot \left(n+2\right)!\right)^{\left(n+2\right)^{n+2}}\cdot \tilde{q}^{2 \cdot\left( \left(n+2\right) \cdot \left(n+1\right)^{n+1}+1\right)}_n \\
& \geq 2^n \cdot \left(2 \pi \cdot n \cdot \left(n+2\right)!\right)^{\left(n+2\right)^{n+1} \cdot \left(n+1\right)} \cdot \tilde{q}^{2 \cdot\left( \left(n+2\right) \cdot \left(n+1\right)^{n+1}+1\right)}_n \\
& \geq 2^n \cdot \left(2 \pi n\right)^{\left(n+2\right) \cdot \left(n+1\right)^{n+1}} \cdot \left(n+2\right)! \cdot \left(\left(n+2\right)!\right)^{\left(n+2\right)^{n-1} \cdot \left(n+1\right)} \cdot \tilde{q}^{2 \cdot\left( \left(n+2\right) \cdot \left(n+1\right)^{n+1}+1\right)}_n \\
& \geq \varphi_1\left(n\right)\cdot \tilde{q}^{2 \cdot\left( \left(n+2\right) \cdot \left(n+1\right)^{n+1}+1\right)}_n
\end{align*}
Hence, the requirement of Theorem \ref{theo:main} is met.

\subsection{Proof of Corollary \ref{cor:KorA}} \label{subsection:prKorA}
Let $\left(\tilde{q}_n\right)_{n \in \mathbb{N}}$ be a sequence satisfying $\tilde{q}_1 \geq \left(\rho+1\right) \cdot 2^7 \cdot \pi^2$ and $\tilde{q}_{n+1} \geq \tilde{q}^{15}_n \cdot \exp\left(\tilde{q}^7_n \cdot \exp\left(\tilde{q}^6_n\right)\right)$. Again we start with a proof by complete induction: \\
\textbf{Claim: } The numbers $\tilde{q}_n$ satisfy $\tilde{q}_n \geq 2^{n+6} \cdot n^2 \cdot \pi^2$. \\
\textbf{Proof: }
\begin{itemize}
	\item \textit{Start $n=1$: } By assumption we have: $\tilde{q}_1 \geq \left(\rho+1\right) \cdot 2^7 \cdot \pi^2 \geq 2^{1+6} \cdot 1^2 \cdot \pi^2$.
	\item \textit{Assumption: } The claim is true for $n \in \mathbb{N}$.
	\item \textit{Induction step $n \rightarrow n+1$: } We estimate
	\begin{equation*}
	\tilde{q}_{n+1} \geq \tilde{q}^{15}_n \cdot \exp\left(\tilde{q}^7_n \cdot \exp\left(\tilde{q}^6_n\right)\right) \geq \left(2^{n+6} \cdot n^2 \cdot \pi^2\right)^{15} \cdot \exp\left(\tilde{q}^7_n \cdot \exp\left(\tilde{q}^6_n\right)\right) \geq 2^{n+7} \cdot \left(n+1\right)^2 \cdot \pi^2.
	\end{equation*}
\end{itemize}
Then we have:
\begin{align*}
	\tilde{q}_{n+1} & \geq \tilde{q}^{15}_n \cdot \exp\left(\tilde{q}^7_n \cdot \exp\left(\tilde{q}^6_n\right)\right) \\
	& \geq  2^{n+6} \cdot n^2 \cdot \pi^2 \cdot \tilde{q}^{14}_n \cdot \exp\left(2^{n+6} \cdot n^2 \cdot \pi^2 \cdot \tilde{q}^6_n \cdot \exp\left(2^{n+6} \cdot n^2 \cdot \pi^2 \cdot \tilde{q}^5_n\right)\right) \\
	& \geq 2^n \cdot n^2 \cdot 64\pi^2 \cdot \tilde{q}^{14}_n \cdot \exp\left(4 \pi \cdot n \cdot \tilde{q}^6_n \cdot \exp\left(2\pi \cdot 2 \cdot n  \cdot \tilde{q}^5_n\right)\right).
\end{align*}
Thus, the condition of Theorem \ref{theo:main2} is fulfilled.

Philipp Kunde \\
University of Hamburg \\
Bundesstraße 55, 20146 Hamburg, Germany \\
Email: Philipp.Kunde@math.uni-hamburg.de
\end{document}